\documentclass[11pt, notitlepage]{article}
\usepackage{amssymb,amsmath,comment}
\catcode`\@=11 \@addtoreset{equation}{section}
\def\thesection{\arabic{section}}

\def\theequation{\thesection.\arabic{equation}}
\catcode`\@=12
\usepackage{colortbl}%
\usepackage{a4wide}

\newcommand{\ds} {\displaystyle}
\newcommand{\e}{\epsilon}
\newcommand{\pa} {\partial}
\newcommand{\al} {\alpha}

\newcommand{\ga} {\gamma}

\newcommand{\Om} {\Omega}
\newcommand{\ra} {\rightarrow}

\newcommand{\De} {\Delta}
\newcommand{\la} {\lambda}
\newcommand{\La} {\Lambda}
\newcommand{\noi} {\noindent}
\newcommand{\na} {\nabla}
\newcommand{\uline} {\underline}
\newcommand{\oline} {\overline}
\newcommand{\mb} {\mathbb}
\newcommand{\mc} {\mathcal}

\usepackage[all]{xy}
\catcode`\@=11
\def\theequation{\@arabic{\c@section}.\@arabic{\c@equation}}
\catcode`\@=12

\def\QED{\hfill {$\square$}\goodbreak \medskip}

\newtheorem{Theorem}{Theorem}[section]
\newtheorem{Lemma}[Theorem]{Lemma}
\newtheorem{Proposition}[Theorem]{Proposition}
\newtheorem{Corollary}[Theorem]{Corollary}
\newtheorem{Remark}[Theorem]{Remark}
\newtheorem{Definition}[Theorem]{Definition}

\begin{document}
{\vspace{0.01in}}

\title
{\sc   On Dirichlet problem for fractional $p$-Laplacian with singular nonlinearity}

\author{
{\bf  Tuhina Mukherjee\footnote{email: tulimukh@gmail.com}}\; and\; {\bf K. Sreenadh\footnote{e-mail: sreenadh@gmail.com}}\\
{\small Department of Mathematics}, \\{\small Indian Institute of Technology Delhi}\\
{\small Hauz Khaz}, {\small New Delhi-16, India}\\
 }

\date{}

\maketitle

\begin{abstract}

\noi In this article, we study the following fractional $p$-Laplacian equation with critical growth singular nonlinearity
\begin{equation*}
 \quad (-\De_{p})^s u = \la u^{-q} + u^{\alpha}, u>0 \; \text{in}\;
\Om,\quad u = 0 \; \mbox{in}\; \mb R^n \setminus\Om.
\end{equation*}
where  $\Om$ is a bounded domain in $\mb{R}^n$ with smooth boundary $\partial \Om$, $n > sp, s \in (0,1), \la >0, 0 < q \leq 1 $ and $\alpha\le p^*_s-1$. We use variational methods to show the existence and multiplicity of positive solutions of above problem with respect to parameter $\la$.
\medskip

\noi \textbf{Key words:}  fractional $p$-Laplacian, Critical exponent, Singular nonlinearities

\medskip

\noi \textit{2010 Mathematics Subject Classification:} 35R11, 35R09,35A15.

\end{abstract}

\section{Introduction}
Let $s \in (0,1)$ and let $\Om \subset \mb R^n$ is a bounded domain with smooth boundary, $n>sp$. We consider the following problem with singular nonlinearity :
\begin{equation*}
(P_{\la}): \quad
 \quad (-\De_{p})^s u = \la u^{-q} + u^{\alpha},\quad  \;u>0 \; \text{in}\;
\Om, \quad u = 0 \; \mbox{in}\; \mb R^n \setminus\Om.
\quad
\end{equation*}
where $\la >0, 0 < q \leq 1 , \alpha \le p^*_s-1,  p^*_s=\frac{np}{n-sp}$ and $(-\De_p)^s$ is the fractional $p$-Laplacian operator defined as
\[(-\De_p)^s u(x) = - 2 \lim_{\epsilon \searrow 0} \int_{\mb R^n \setminus B_{\epsilon}(x)} \frac{|u(x)-u(y)|^{p-2}(u(x)-u(y))}{|x-y|^{n+sp}}dy \; \text{for all }\; x \in \mb R^n. \]
Recently a lot of attention is given to the study of fractional and non-local operators of elliptic type due to concrete real world applications in finance, thin obstacle problem, optimization, quasi-geostrophic flow etc.

\noi Semilinear Dirichlet  problem for  fractional Laplacian using variational methods is recently studied in \cite{XcT, s1, s2}.
The existence and multiplicity results for
non-local operators like fractional Laplacian with combination of convex and concave type non linearity
like $u^q+\la u^p, p,q>0$ is studied in  \cite{bss,bc,mb,mb1,XsY,xy}.
Eigenvalue problem for the fractional $p$-Laplacian and properties like simplicity of smallest eigenvalue is
studied in \cite{EPL,GG}. The Brezis-Nirenberg type existence result is studied in \cite{s3}. Existence results with convex-concave type regular nonlinearities is studied in \cite{ss1}.

\medskip
\noi In the local setting ($s=1$), the paper by Crandal, Rabinowitz and Tartar \cite{crt} is the starting point on semilinear problem with singular nonlinearity. A lot of work has been done related to existence and multiplicity results for Laplacian and $p$-Laplacian with singular non-linearity, see \cite{AJ,GST,GST2,diaz,coc,GR1,GR2}. {}{ In \cite{diaz,coc}, the authors studied the singular problems of the type
\[-\Delta u= g(x,u)+ h(x,\la u),\; \text{in}\; \Om,\quad u=0\; \text{ on}\;  \partial \Om, g(x,u)\in L^1(\Om)\]
with $g(x,u)\sim u^{-\al}$. They studied the existence of solutions under suitable conditions on $g$ and $h$.
In \cite{GR1} and \cite{GR2}, authors conside the singular problems of the type
\[ -\Delta u + K(x)g(u)= \la f(x,u)+ \mu h(x) \; \text{ in } \Om, \; u=0 \; \text{ on } \partial \Om, \]
where $\Om$ is smooth bounded domain in $\mathbb R^n$, $n\geq 2$ and $\la >0$. Here, $h,K \in C^{0,\gamma}(\Om)$ for some $0< \gamma<1$ and $h>0$ in $\Om$, $f : [0, \infty) \rightarrow [0, \infty)$ is a H\"{o}lder continuous function which is positive on $\bar{\Om} \times (0, \infty)$ that is sublinear at $\infty$ and of superlinear at $0$.
 The function $g \in C^{0,\gamma}(0,\infty)$ for some $0<\gamma<1$ is non negative and non increasing such that $\lim\limits_{s \rightarrow 0^+ } g(s) = +\infty$. They proved several results related to existence and non existence of positive solutions of above problems taking into account both the sign of the potential $K$ and the decay rate around the origin of the singular nonlinearity $g$. Several authors conside the problems of Lane-Emden-Fowler type with singular nonlinearity such as \cite{louis,florica,GR3}. 
In addition, some bifurcation results  has been proved in \cite{GR3} for the problem
\[-\Delta u = g(u) + \la|\nabla u|^p + \mu f(x, u)\; \text{in } \Om,\; u > 0\; \text{in } \Om,\; u = 0\; \text{on } \partial \Om,\]
where $\la,\mu \geq 0$, $0 < p \leq 2$, $f$ is non-decreasing with respect to the second variable and  $g(u)$ behaves like $u^{-\alpha}$ around the origin.   The asymptotic behaviour of the solutions is shown by constructing suitable sub- and supersolutions combined with the maximum principles. We also refer \cite{GHS, HM} as a part of previous contributions to this field. For detailed study and recent results on singular problems we refer to  \cite{book1}.}

\noi In \cite{GST}, authors studied the critical growth singular problem
\[
-\De_{p} u = \la u^{-\delta}+ u^{q}, \quad  u>0 \; \text{in}\; \Omega, \quad u=0 \; \text{on}\; \pa \Om
\]
where $0<\delta<1$ and $p-1 < q \leq p^*-1 $ and $\De_p u= div(|\nabla u|^{p-2} \na u)$. Using the variational methods, they proved the existence of multiple solutions with restriction on $p$ and $q$ in the spirit of \cite{AGP2,Gossez}. Among the works dealing with elliptic equations with singular and critical growth terms, we cite also  \cite{AJ,Am,Sk,GF,DSS} and references there-in, with no attempt to provide a complete list.

\noindent Recently, the study of the fractional elliptic equations attracted  lot of interest by researchers in nonlinear analysis.
There are many works on existence of a solution for fractional elliptic equations with regular nonlinearities like $u^q+\la u^p, \; p,\; q>0$. The sub critical growth problems are studied in  \cite{XcT,s1,s2}  and critical exponent problems are studied in \cite{bc,mb,mb1,s3}. Also, the multiplicity of solutions by  the method of Nehari manifold and fibering maps has been investigated in \cite{ss1,xy,zlh}. For detailed study and recent results on this subject we refer to \cite{book2}.
\noi In \cite{peral} the authors the singular problem
\begin{equation*}
(-\De)^s u = \la\frac{f(x)}{u^\ga} + M u^{p}, \;
 u>0\;\text{in}\;\Om, \quad
 u = 0 \; \mbox{in}\; \mb R^n \setminus\Om,
\end{equation*}
where  $n>2s$, $M\ge 0$, $0<s<1$, $\ga>0$, $\la>0$, $1<p<2_{s}^{*}-1$ and $f\in L^{m}(\Om)$, $m\geq 1$ is a nonnegative function. Here authors studied the existence of distributional solutions for small $\la$ using the uniform estimates of  $\{u_n\}$ which are solutions of the regularized problems with singular  term $u^{-\ga}$ replaced by $(u+\frac{1}{n})^{-\ga}$. In \cite{TS}, the critical ($p=2_{s}^{*}-1$) singular problem is studied where  multiplicity results are obtained using the Nehari manifold approach.

\noindent There are many works on the study of  $p$-fractional equations with polynomial type nonlinearities. In cite{ss1} authors studied the subcritical problems using Nehari manifold and fibering maps. In \cite{s3}, Brezis-Nirenberg type critical exponent problem is studied. We also \cite{bms,ss2,Asm} and references therein.
To the best of our knowledge, there are no works on existence or multiplicity results with singular nonlinearities.

\medskip
\noi In this paper, we study the existence and multiplicity results with convex-concave type singular nonlinearity. Here we follow the approach as
 in \cite{Hirano}. We obtain our results by studying the existence of minimizers that arise out of structure of Nehari manifold. We would like to remark
 that the results proved here are new even for the case $q=1$. Also the existence result is sharp in the sense that we show the existence
 of $\Lambda$ such that  $(0,\Lambda)$ is the maximal range for $\lambda$ for which the solution exists. We show the existence of second solution in the sub-critical case for suitable range of $\lambda$ where the fibering maps has two critical points. We also show some regularity results on weak solutions.

\medskip
\noi The paper is organized as follows: In section 2, we present some preliminaries on function spaces requi for variational settings. In section 3, we study the corresponding Nehari manifold and properties of minimizers. In section 4 and 5,  we show the existence of minimizers and solutions and state the main results. In section 6, we show some regularity results and section 7 is devoted to the maximal range of $\lambda$ for existence of solutions.
\section{Preliminaries and Main Results}
\noi In \cite{ss1}, authors  discussed the Dirichlet
boundary value problem involving $p$-fractional Laplace operator using the variational techniques.  Due to non-localness of the fractional
Laplacian, they introduced the function space $(X_0,\|.\|_{X_0})$.
The space $X$ is defined as
\[X= \left\{u|\;u:\mb R^n \ra\mb R \;\text{is measurable},\;
u|_{\Om} \in L^p(\Om)\;
 \text{and}\;  \frac{(u(x)- u(y))}{ |x-y|^{\frac{n+sp}{p}}}\in
L^p(Q)\right\},\]
\noi where $Q=\mb R^{2n}\setminus(\mc C\Om\times \mc C\Om)$ and
 $\mc C\Om := \mb R^n\setminus\Om$. The space X is endowed with the norm
\begin{align*}
 \|u\|_X = \|u\|_{L^p(\Om)} + \left[u\right]_X, \quad \text{where}\; \left[u\right]_X= \left( \int_{Q}\frac{|u(x)-u(y)|^{p}}{|x-y|^{n+sp}}dx
dy\right)^{\frac{1}{p}}.
\end{align*}
 Then we define $ X_0 = \{u\in X : u = 0 \;\text{a.e. in}\; \mb R^n\setminus \Om\}$. Also, there exists a constant $C>0$ such that $\|u\|_{L^{p}(\Om)} \le C [u]_X$ for all $u\in X_0$. Hence,  $\|u\|=[u]_X$ is a norm on $X_0$ and  $X_0$ is a Hilbert space. Note that the norm $\|.\|$ involves the interaction between $\Om$ and $\mb R^n\backslash\Om$. We denote $\|.\|_{L^p(\Om)}$ as $|.|_p$ and $\|.\|=[.]_X$ for the norm in $X_0$. Now for each $\beta \geq 0$, we set
\begin{equation}\label{eq00}
 C_{\beta} = \sup \left \{  |u|_{\beta}^{\beta} \; :\;u \in X, \; \|u\| = 1\right \}.
 \end{equation}
Then $C_0  = |\Om| $ = Lebesgue measure of $\Om$ and
$\int_{\Om} |u|^\beta dx \leq C_\beta \|u\|^{\beta}$, for all $ u \in X_0$. From the embedding results in \cite{ss1}, we know that $X_0$ is continuously and compactly embedded in $L^r(\Om)$  where $1\leq r < p^*_s$ and the embedding is continuous but not compact if $r= p^*_s$. We define the best constant of the embedding $S$ as
 \begin{equation*}
 S = \inf \{ \|u\|^p \; :\; u \in X_0,\; |u|^p_{p^*_s}=1\}.
 \end{equation*}

\begin{Definition}
We say $u\in X_0$ is a positive weak solution of $(P_\la)$
if $u>0$ in $\Om$ and
$$ \int_Q \frac{|u(x)-u(y)|^{p-2}(u(x)-u(y))(\psi(x)-\psi(y))}{|x-y|^{n+sp}} ~dxdy -  \int_\Om \left( \la u^{-q}  -  u^{\alpha}\right) \psi ~dx = 0 $$ for all $\psi \in C^{\infty}_c(\Om).$
\end{Definition}
We define the  functional associated to $(P_\la)$ as $I_{\la} : X_{0} \rightarrow (-\infty, \infty]$ as
\[ I_{\la}(u) = \frac{1}{p} \int_Q \frac{|u(x) - u(y)|^p}{|x-y|^{n+sp}}dxdy - \la \int_\Om  G_q(u) dx - \frac{1}{\alpha+1} \int_\Om |u|^{\alpha+1}dx \]
where $G_q : \mb{R} \rightarrow [-\infty, \infty)$ is the function defined by
$$G_q(x)=
\left\{
	\begin{array}{ll}
		\frac{|x|^{1-q}}{1-q}  & \mbox{if } 0<q<1 \\
		\ln|x| & \mbox{if } q=1
	\end{array}
\right.$$
for $x \in \mb{R}$. For each $0 < q \leq 1$, we set
$\ds X_+ = \{ u \in X_0 : u \geq 0\}$ and
$$X_{+,q} = \{ u \in X_+ : u \not\equiv 0, \; G_q(u) \in L^1(\Om)\}.$$
Notice that $ X_{+,q} = X_+ \setminus \{0\}$ if $0<q<1$ and $X_{+,1} \neq \emptyset$ if $\partial \Om$ is, for example, of $C^2$. We will need the following important Lemma.
\begin{Lemma}\label{lem2.1}
For each $w \in X_{+}$, there exists a sequence $\{w_k\}$ in $X_{0}$ such that,$w_k \rightarrow w$ strongly in $X_0$, where $0 \leq w_1 \leq w_2 \leq \ldots$ and $w_k$ has compact support in $\Om$, for each k .
\end{Lemma}
\begin{proof} Proof here is adopted from \cite{Hirano}.
Let $w \in X_+$ and $\{\psi_k\}$ be sequence in $C^{\infty}_{c}(\Om)$ such that $\psi_k$ is non negative and converges strongly to $w$ in $X_0$. Define $z_k = \min \{\psi_k, w\}$,  then $z_k \rightarrow w$ converges strongly to $w$ in $X_0$. Now, we set $w_1 = z_{r_1}$ where $r_1>0$ is such that $\|z_{r_1}-w\| \leq 1$. Then $\max \{w_1, z_m\} \rightarrow w$ strongly as $m \rightarrow \infty$, thus we can find $r_2>0$ such that $\|\max\{w_1,z_{r_2}\}-w\| \leq 1/2$. We set $w_2=\max\{w_1,z_{r_2}\}$ and get $\max \{w_2, z_m\} \rightarrow w$ strongly as $m \rightarrow \infty$. Consequently, by induction we set, $w_{k+1}= \max\{w_k,z_{r_{k+1}}\}$ to obtain the desi sequence, since we can see that $w_k \in X_0$  has compact support, for each $k$ and $\|\max\{w_k,z_{r_{k+1}}\}-w\| \leq 1/(k+1)$ which says that $\{w_k\}$ converges strongly to $w$ in $X_0$ as $k \rightarrow \infty$. \QED
\end{proof}

\noi Let $\phi_1>0$ be the eigenfunction of $(-\De_{p})^s$ corresponding to the smallest eigenvalue $\la_1$. This is obtained as minimizer of the minimization problem
\[\la_1= \min \{\|u\|\; :\; u\in X_0, \;\; \|u\|_{L^p(\Om)}=1\}.\]
In (see \cite{s3,EPL}) it was shown that this minimizer is achieved by unique positive and bounded function $\phi_1$. Moreover $(\la_1, \phi_1)$ is  the solution of the eigenvalue problem
\begin{equation*}
 \quad (-\De_{p})^s u = \la_1 |u|^{p-2}u, \; u>0\; \text{in}\;
\Om,\quad   u = 0 \; \mbox{on}\; \mb R^n \setminus\Om.
\end{equation*}
We assume $\|\phi_1\|_{L^\infty} = 1$. With these preliminaries, we state our main results.
\medskip

\noi For each $u\in X_{+,q}$ we define the fiber map $\phi_u:\mb R^+ \rightarrow \mb R$ by $\phi_u(t)=I_\la(tu)$. Then we prove
\begin{Theorem} \label{thm3.2}
Assume $0<q \leq 1$. In case $q=1$, assume also $X_{+,1} \neq \emptyset$. Let $\Lambda_1$ be a constant defined by
 $\Lambda_1 = \sup \left\{\la >0:\text{ for each} ~ u\in X_{+,q}\setminus\{0\}, ~\phi_u(t)~  \text{has two critical points in}  ~(0, \infty) \right\}.$
Then $\Lambda_1 >0$.
\end{Theorem}
\begin{Theorem}\label{thm2.4}
For all $\la \in (0, \Lambda_1)$, $(P_\la)$ has at least two distinct solutions in $X_{+,q}$ when $\alpha < p^*_s-1$ and at least one solution in the critical case $\alpha = p^*_s-1$.
\end{Theorem}

\begin{Definition}
We say $u \in X_0$ a weak sub solution of $(P_\la)$ if $u >0$ in $\Om$ and
$$ \int_Q \frac{|u(x)-u(y)|^{p-2}(u(x)-u(y))(\psi(x)-\psi(y))}{|x-y|^{n+sp}} ~dxdy \leq  \int_\Om \left( \la u^{-q}  +  u^{\alpha}\right) \psi ~dx = 0 $$
for all $0 \leq \psi \in C^{\infty}_c(\Om)$. Similarly $u \in X_0$ is said to be a weak super solution to $(P_\la)$ if in the above the reverse inequalities hold.
\end{Definition}

Next we study that the existence of solution with the parameter in maximal interval. For this we minimize the functional over the convex set $\{u\in X_{+,q}: \uline{u}\le u\le \oline{u}\}$ where $\uline{u}$ and $\oline{u}$ are sub and super solutions respectively. Using truncation techniques as in \cite{yh}, we show that the minimizer is a solution.
\begin{Theorem}\label{thm2.5} Let $\alpha\le p_s^*-1$ and $0<q\le 1$. Then there exists $\La>0$
such that $(P_\la)$ has a solution for all $\la\in (0,\La)$ and no solution for $\la>\La$.
\end{Theorem}

\section{Nehari manifold and fibering maps}
We denote $I_{\la} = I$ for simplicity now. One can easily verify that the energy functional $I$ is not bounded below on the space $X_0$. We will show that it is bounded on the manifold associated to the functional $I$. In this section, we study the structure of this manifold.  We define
\[ \mc N_{\la} = \{ u \in X_{+,q} | \left\langle I^{\prime}(u),u\right\rangle = 0 \}. \]
\begin{Theorem}
$I$ is coercive and bounded below on $\mc N_{\la}$.
\end{Theorem}
\begin{proof}
In case of $0<q<1$, since $u\in \mc N_\la$, using the embedding of $X_0$ in $L^{1-q}(\Om)$, we get
\begin{equation*}
\begin{split}
I(u) & = \left(\frac{1}{p}-\frac{1}{\alpha+1}\right)\|u\|^p- \la \left(\frac{1}{1-q}-\frac{1}{\alpha+1}\right)\int_{\Om} |u|^{1-q}dx\\
& \geq c_1 \|u\|^p - c_2 \|u\|^{1-q}
\end{split}
\end{equation*}
for some constants $c_1$ and $c_2$. This says that $I$ is coercive and bounded below on $\mc N_{\la}$.\\
\noi In case of $q=1$, using the inequality $\ln|u| \leq |u|$  and embedding results for $X_0$, we can similarly get $I$ as bounded below.\QED
\end{proof}

\noi From the definition of fiber map $\phi_u$, we have
$$\phi_u(t)=
\begin{cases} \ds
		\frac{t^p}{p} \|u\|^p - \frac{t^{1-q}}{1-q} \int_{\Om} |u|^{1-q} dx - \frac{t^{\alpha+1}}{\alpha+1} \int_{\Om} |u|^{\alpha+1} dx & \; \text{if}\; ~  0<q<1 \\
		\ds \frac{t^p}{p} \|u\|^p - \la \int_{\Om} \ln(t|u|) dx - \frac{t^{\alpha+1}}{\alpha+1} \int_{\Om} |u|^{\alpha+1} dx  & \; \text{if}\;~ q=1.
	\end{cases}$$
which gives
\[ \phi^{\prime}_u(t) = t^{p-1} \|u\|^p -\la t^{-q} \int_{\Om} |u|^{1-q} dx - t^{\alpha} \int_{\Om}|u|^{\alpha+1;} dx \]
\[ \phi^{\prime \prime}_u(t) = (p-1)t^{p-2}\|u\|^p + q \la t^{-q-1} \int_{\Om}|u|^{1-q} dx - \alpha t^{\alpha-1} \int_{\Om}|u|^{\alpha+1} dx .\]
It is easy to see that the points in $\mc N_{\la}$ are corresponding to critical points of $\phi_{u}$ at $t=1$. So, it is natural to divide $\mc N_{\la}$ into three sets corresponding to local minima, local maxima and points of inflexion. Therefore, we define
\begin{align*}
\mc N_{\la}^{+} = & \{ u \in \mc N_{\la}|~ \phi^{\prime}_u (1) = 0,~ \phi^{\prime \prime}_u(1) > 0\}=  \{ t_0u \in \mc N_{\la} |\; t_0 > 0,~ \phi^{\prime}_u (t_0) = 0,~ \phi^{\prime \prime}_u(t_0) > 0\}\\
\mc N_{\la}^{-} = & \{ u \in \mc N_{\la} |~ \phi^{\prime}_u (1) = 0,~ \phi^{\prime \prime}_u(1) <0\}=  \{ t_0u \in \mc N_{\la} |\; t_0 > 0, ~ \phi^{\prime}_u (t_0) = 0, ~\phi^{\prime \prime}_u(t_0) < 0\}
\end{align*}
and, $ \mc N_{\lambda}^{0}= \{ u \in \mc N_{\la} | \phi^{\prime}_{u}(1)=0, \phi^{\prime \prime}_{u}(1)=0 \}. $

\begin{Lemma}\label{lem3.2}
There exists $\la_*>0$ such that for each $u\in X_{+,q}\backslash\{0\}$, there is unique $t_{\max}, t_1$ and $t_2$ with property that $t_< t_{max}<t_2$, $t_1 u\in \mc N_{\la}^{+}$ and $ t_2 u\in \mc N_{\la}^{-}$ and  for all $\la \in (0,\la_*)$.
\end{Lemma}
\begin{proof}
Define $A(u)= \int_{\Om}|u|^{1-q}dx$ and $B(u)= \int_{\Om}|u|^{\alpha+1}dx$. Let $u \in X_{+,q}$ then we have
\begin{align*}
\frac{d}{dt}I(tu)
=& t^{p-1}\|u\|^p - t^{-q}\la  A(u) - t^{\alpha} B(u)\\
=&t^{-q} \left (m_u(t) - \la A(u) \right )
\end{align*}
and we define $m_u(t) := t^{p-1+q} \|u\|^p - t^{\alpha+q} B(u)$. Since $\ds \lim_{t \rightarrow \infty} m_u(t) = - \infty$,
we can easily see that $m_u(t)$ attains its maximum at $t_{max} = \left [ \frac{(p-1+q)\|u\|^p}{(\alpha+q) B(u)} \right]^{\frac{1}{\alpha+1-p}} $ and
\[ m_u(t_{max}) = \left( \frac{\alpha+2-p}{p-1+q} \right) \left( \frac{p-1+q}{\alpha+q} \right)^{\frac{\alpha+q}{\alpha+1-p}} \frac{\|u\|^\frac{p(\alpha+q)}{\alpha+1-p}}{B(u)^\frac{p-1+q}{\alpha+1-p}}. \]
Now, $u \in \mc N_{\la}$  if and only if  $m_u(t) =  \la A(u) $ and we see that
\begin{equation*}
\begin{split}
m_u(t_{max}) - \la A(u)
\geq ~&  m_u(t_{max}) - \la |u|^{1-q}_{{1-q}}\\
\geq ~& \left( \frac{\alpha+2-p}{p-1+q} \right) \left( \frac{p-1+q}{\alpha+q} \right)^{\frac{\alpha+q}{\alpha+1-p}} \frac{\|u\|^\frac{p(\alpha+q)}{\alpha+1-p}}{B(u)^\frac{p-1+q}{\alpha+1-p}} - \la C_{1-q} \|u\|^{1-q}>0\\
\end{split}
\end{equation*}
if and only if $\la < \left( \frac{\alpha+2-p}{p-1+q} \right) \left( \frac{p-1+q}{\alpha+q} \right)^{\frac{\alpha+q}{\alpha+1-p}} (C_{\alpha+1})^{\frac{-p+1-q}{\alpha+1-p}}C_{1-q}^{-1} $(say), where $C_{\beta}$ is defined as in \eqref{eq00}.

 \noi Case(I) $(0<q<1)$: We can also see that $ m_u(t) = \la \int_{\Om} |u|^{1-q}dx$ if and only if $\phi^{\prime}_u (t) = 0$. So for $\la \in (0,\lambda_*)$, there exists exactly two points $0<t_1< t_{max}<t_2$ with $m^{\prime}_u(t_1)>0$ and $m^{\prime}_u(t_2)<0$ that is, $t_1u \in \mc N^{+}_{\la}$ and $t_2u \in \mc N^{-}_{\la}$. Thus, $\phi_u$ has local minimum at $t=t_1$ and local maximum at $t=t_2$, that is $\phi_{u}$ is decreasing in $(0,t_1)$ and increasing in $(t_1,t_2)$. 

\noi Case(II)$(q=1)$: Since $\ds \lim_{t \rightarrow 0} \phi_{u}(t) = \infty$ and $\ds \lim_{t \rightarrow \infty} \phi_{u}(t) = - \infty$ with similar reasoning as above we get $t_1, t_2$.  
That is, in both cases, $\phi_{u}$ has exactly two critical points $t_1$ and $t_2$ such that $0< t_1 <t_{max}< t_2$, $\phi^{\prime \prime}_{u}(t_1) > 0$ and  $\phi^{\prime \prime}_{u}(t_2) < 0$ that is $t_1u \in \mc N_{\la}^{+}$, $t_2u \in \mc N_{\la}^{-}$.\QED
\end{proof}
{\bf Proof of Theorem \ref{thm3.2}}:
From Lemma \ref{lem3.2}, we see that $\Lambda_1$ is positive. If $I_\la(tu)$ has two critical points for some $\lambda=\lambda^*$, then $t\mapsto I_\la(tu)$ also has two critical points for all $\la<\la^*$.\QED
\begin{Corollary}
$\mc N_{\la}^{0} = \{0\}$ for all $ \la \in (0, \La_1)$.
\end{Corollary}
\begin{proof}
Let $u  \in \mc N_{\la}^{0}$ and $u\not\equiv 0$. Then $u \in \mc N_{\la}.$ That is, $t=1$ is a critical point of $\phi_{u}(t)$. By Lemma \ref{lem3.2},  $\phi_{u}$ has critical points corresponding to either local minima or local maxima. So, $t=1$ is the critical point corresponding to either local minima or local maxima of $\phi_{u}$. Thus, either $u \in \mc N_{\la}^{+}$ or $u \in \mc N_{\la}^{-}$, which is a contradiction.\QED
\end{proof}
\noi We can now show that $I$ is bounded below on $\mc N_{\la}^{+}$ and $\mc N_{\la}^{-}$ in following way:
\begin{Lemma}\label{le01}
$\inf I(\mc N_{\la}^{+}) > - \infty$ and $\inf I(\mc N_{\la}^{-}) > - \infty$.
\end{Lemma}
\begin{proof}
Let $u \in \mc N_{\la}^{+}$ and $v\in \mc N_{\la}^{-}$. Then we have
\begin{equation*}
\begin{split}
0  < \phi^{\prime \prime}_u(1) & \leq (p-1-\alpha) \|u\|^p + \la (\alpha+q) C_{1-q} \|a\|_{\infty} \|u\|^{1-q} ,\\
0  > \phi^{\prime \prime}_v(1) & \geq (p-1+q) \|v\|^p - (\alpha+q) C_{\alpha+1}\|v\|^{\alpha+1}.
\end{split}
\end{equation*}
 Thus we obtain
\[ \|u\| \leq \left ( \frac{\la (\alpha+q) C_{1-q} }{\alpha+1-p} \right )^{\frac{1}{p+q-1}}\; \text{and} \;  \|v\| \geq \left ( \frac{p-1+q}{(\alpha+q)C_{\alpha+1}}\right )^{\frac{1}{\alpha+1-p}}.\]
This implies that
 \begin{equation}\label{eq3n1}
 \sup \{ \|u\| : u \in \mc N_{\la}^{+}\} < \infty \; \text{ and }\;  \inf \{ \|v\| : v \in \mc N_{\la}^{-} \} >0.
\end{equation}
If $I(v) \leq M$, using $\ln(|v|) \leq |v|$ we get
\begin{equation}\label{eq3n2}
\begin{array}{rllll}
\displaystyle \frac{\alpha+1-p}{p(\alpha+1)}  \|v\|^p -  \frac{\la (\alpha+q)C_{1-q}}{(\alpha+1)(1-q)}  \|v\|^{1-q} \leq M, & 0<q<1\\
\displaystyle \text{and } \frac{\alpha+1-p}{p(\alpha+1)} \|v\|^p - \la C_1 \|v\|+ \frac{\la}{\alpha+1}\le M, & q=1.
\end{array}\end{equation}
 which implies  $ \sup \{ \|v\| : v \in \mc N_{\la}^{-} , I(v) \leq M\} < \infty$ for each $M > 0$.
 Using \eqref{eq3n1} and \eqref{eq3n2}, it is easy to show that $\inf I(\mc N_{\la}^{+}) > - \infty$ and  $\inf I(\mc N_{\la}^{-}) > - \infty$. \QED
\end{proof}


\begin{Lemma}\label{le03} Suppose $u\in \mc N_{\la}^{+}$ and $v\in \mc N_{\la}^{-}$ be minimizers of $I$ over $\mc N_{\la}^{+}$ and $\mc N_{\la}^{-}$ respectively. Then for each $ w \in X_{+}$,
\begin{enumerate}
\item  there exists $\epsilon_0 > 0$ such that $I(u +\epsilon w) \geq I(u)$ for each $ \epsilon \in [0, \epsilon_0]$
\item $t_{\epsilon} \rightarrow 1$  as $\epsilon \rightarrow 0^+$, where $t_{\epsilon}$ is the unique positive real number satisfying $t_{\epsilon} (v + \epsilon w) \in \mc N_{\la}^{-}.$
\end{enumerate}
\end{Lemma}
\begin{proof}
\begin{enumerate}
\item{
Let $w \in X_{+}$ that is $w \in X_0$ and $w \geq 0$. We set
$$\rho(\epsilon) = (p-1)\|u+\epsilon w\|^p + \la q \int_{\Om} |u+\epsilon w|^{1-q}dx - \alpha \int_{\Om} |u+\epsilon w|^{\alpha+1}dx$$
for each $\epsilon \geq 0$. Then using continuity of $\rho$, $\rho(0) = \phi^{\prime \prime}_u(1) >0$ and $u \in \mc{N}^{+}_{\la}$, there exist $\epsilon_0>0 $ such that $\rho(\epsilon)> 0$ for $\epsilon \in [0, \epsilon_0]$. Since for each $\epsilon > 0$, there exists $t_{\epsilon}^{\prime}>0$ such that $t_{\epsilon}^{\prime}(u + \epsilon w) \in \mc{N}^{+}_{\la}$. So, $t_{\epsilon}^{\prime} \rightarrow 1$ as $\epsilon \rightarrow 0$ and for each $\epsilon \in [0, \epsilon_0]$ we have
\[ I(u + \epsilon w) \geq I(t_{\epsilon}^{\prime}(u + \epsilon w))\geq \inf I(\mc{N}^{+}_{\la})= I(u).\]
}
\item{
We define $h :(0, \infty)\times \mb R^3 \rightarrow \mb R  $ by
\[ h(t,l_1,l_2,l_3) = l_1t^{p-1} - \la t^{-q}l_2 - t^{\alpha}l_3 \]
for $(t,l_1,l_2,l_3)\in (0, \infty)\times \mb R^3.$ Then $h$ is $C^{\infty}$ function. Then, we have
\begin{equation*}
\frac{dh}{dt}(1, \|v\|^p,\int_{\Om}|v|^{1-q}dx ,\int_{\Om}|v|^{\alpha+1}) = \phi^{\prime \prime}_v(1)<0,
\end{equation*}
and for each $\epsilon \geq 0, \; h(t_{\epsilon}, \|v+\epsilon w\|^p, \int_{\Om}|v + \epsilon w|^{1-q}dx ,\int_{\Om} |v|^{\alpha+1}) = \phi^{\prime}_{v+\epsilon w}(t_\epsilon)=0$. Also
\[ h(1, \|v\|^p, \int_{\Om} |v|^{1-q}dx, \int_{\Om} |v|^{\alpha+1}) = \phi^{\prime}_v(1) = 0.\]
Therefore, by implicit function theorem, there exists an open neighborhood $ A \subset (0, \infty)$ and $B \subset \mb R^3$ containing $1$ and $(\|v\|^p,\int_{\Om} |v|^{1-q}dx, \int_{\Om}|v|^{\alpha+1} )$ respectively such that  for all $y \in B$, $h(t,y) = 0 $ has a unique solution $ t = g(y)\in A $, where $g : B \rightarrow A$ is a continuous function. So, $(\|v+\epsilon w\|^p,\; \int_{\Om}|v+ \epsilon w|^{1-q} dx, \int_{\Om}|v+ \epsilon w|^{\alpha+1}) \in B$ and
\[ g \left( \|v+\epsilon w)\|^p,\; \int_{\Om}|v+ \epsilon w|^{1-q} dx, \int_{\Om}|v+ \epsilon w|^{\alpha+1} \right) = t_{\epsilon}  \]
\noi since $h(t_{\epsilon},\|v+\epsilon w)\|^p,\; \int_{\Om}|v+ \epsilon w|^{1-q} dx, \int_{\Om}|v+ \epsilon w|^{\alpha+1}) = 0$. Thus, by continuity of $g$, we get $t_{\epsilon} \rightarrow 1$ as $\epsilon \rightarrow 0^+$.}\QED
\end{enumerate}
\end{proof}

\begin{Lemma} \label{lem3.6} Suppose $u\in \mc N_{\la}^{+}$ and $v\in N_{\la}^{-}$ are minimizers of $I$ on $\mc N_{\la}^{+}$ and $\mc N_{\la}^{-}$ respectively. Then
for each $w \in X_{+}$, we have $u^{-q}w, v^{-q} w \in L^{1}(\Om)$ and
\begin{align}
&\int_Q \frac{|u(x)-u(y)|^{p-2}(u(x)-u(y))(w(x)-w(y))}{|x-y|^{n+sp}}~ dxdy - \la \int_\Om \left( u^{-q}+  u^{\alpha}\right)w dx \geq 0, \label{eq4.2}\\
&\int_Q \frac{|v(x)-v(y)|^{p-2}(v(x)-v(y))(w(x)-w(y))}{|x-y|^{n+sp}} ~dxdy - \la \int_\Om  \left(v^{-q} +  v^{\alpha}\right)w dx \geq 0.\label{eq4.3}\end{align}
\end{Lemma}
\begin{proof}
Let $w \in X_{+}$. For sufficiently small $\epsilon > 0$, by Lemma \ref{le03},
\begin{equation}\label{eq7}
\begin{split}
0  \leq \frac{I(u+\epsilon w) - I(u)}{\epsilon}
 = & \frac{1}{p\e} (\|u+\epsilon w\|^p- \|u\|^p)- \frac{\la}{\e} \int_{\Om} (G_q(u + \epsilon w) - G_q(u)) dx \\
& - \frac{1}{\e (\alpha+1)} \int_{\Om} (|u+\epsilon w|^{\alpha+1} - |u|^{\alpha+1}) dx \\
\end{split}
\end{equation}
We can easily verify that as $\epsilon \rightarrow 0^+$,
\begin{enumerate}
\item[($i$)] $\ds \frac{(\|u+ \epsilon w\|^p- \|u\|^p)}{\epsilon} \rightarrow p\int_{Q} \frac{|u(x)-u(y)|^{p-2}(u(x)-u(y))(w(x)-w(y))}{|x-y|^{n+sp}}~dxdy$
\item[($ii$)] $\ds \int_{\Om} \frac{(|u+\epsilon w|^{\alpha+1}-|u|^{\alpha+1})}{\epsilon} dx \rightarrow (\alpha+1) \int_{\Om} |u|^{\alpha-1}u w dx.$
\end{enumerate}
which implies that $ \frac{(G_q(u + \epsilon w) - G_q(u))}{\epsilon} \in L^{1}(\Om)$. Also, for each $x \in \Om,$
$$\frac{G_q(u(x)+\epsilon w(x)) - G_q(u(x))}{\epsilon}=
\left\{
	\begin{array}{ll}
		\frac{1}{\epsilon} \left( \frac{|u+\epsilon w|^{1-q}(x) - |u|^{1-q}(x)}{1-q} \right) & \mbox{if } 0<q<1 \\
		\frac{1}{\epsilon} \left( \ln(|u+\epsilon w|) - \ln(|u|) \right) & \mbox{if } q=1
	\end{array}
\right.$$
which increases monotonically as $\epsilon \downarrow 0$ and
$$\lim\limits_{\epsilon \downarrow 0} \frac{G_q(u(x)+\epsilon w(x)) - G_q(u(x))}{\epsilon} =
\left\{
	\begin{array}{ll}
	0 & \mbox{if} \; w(x)=0 \\
	(u(x))^{-q} w(x) & \mbox{if} \; w(x)>0 , u(x) > 0\\
	\infty & \mbox{if}\; w(x) > 0 , u(x) =0.
    \end{array}
\right.$$
So using monotone convergence theorem for $\{G_q\}$, we get $u^{-q}w \in L^1(\Om)$. Letting $\epsilon \downarrow 0$ in both sides of \eqref{eq7}, we get \eqref{eq4.2}.
Next, we will show these properties for $v$. For each $\epsilon >0 $, there exists $t_{\epsilon}>0$ with $t_{\epsilon}(v+\epsilon w) \in \mc N^-_\la$. By Lemma \ref{le03}(2), for sufficiently small $\epsilon > 0$, there holds
\[ I(t_{\epsilon}(v+\epsilon w)) \geq I(v) \geq I(t_{\epsilon}v)\]
which implies $I(t_{\epsilon}(v+\epsilon w)) - I(v) \geq 0$ and thus, we have
\begin{align*}
\la  \int_{\Om} ( G_q(t_{\epsilon}|v+\epsilon w|^{1-q}) - G_q(|v|^{1-q})) dx  \leq &\frac{t_{\epsilon}^{p}}{p} (\|v+\epsilon w\|^p - \|v\|^p) \\
&- \frac{t^{\alpha+1}_\epsilon}{ \alpha+1} \int_{\Om}(|v+\epsilon w|^{\alpha+1} - |v|^{\alpha+1})dx.\end{align*}
As $\epsilon \downarrow 0$, $t_\epsilon \rightarrow 1$. Thus, using similar arguments as above, we obtain $v^{-q}w \in L^1(\Om)$ and \eqref{eq4.3} follows. \QED
\end{proof}

\noi Let $\eta > 0$ be such that $\phi = \eta \phi_1$ satisfies
 \begin{equation} \label{eq3.4}
 \int_{Q}\frac{|\phi(x)-\phi(y)|^{p-2}(\phi(x)-\phi(y))(\psi(x)-\psi(y))}{|x-y|^{n+sp}} dxdy \leq \la \int_{Q}\phi^{-q}\psi+\int_{Q}\phi^{\alpha}\psi
 \end{equation}
for all $\psi \in X_0$ (i.e $\phi$ is a sub-solution of $(P_{\la})$ ) and  $\phi^{\alpha+q} (x) \leq \la \left(\frac{q}{\alpha}\right)$, for each $x \in \Om$. Then we have
 \begin{Lemma}\label{le05} Suppose $u\in \mc N_{\la}^{+}, v\in \mc N_{\la}^{-}$ are minimizers of $I$ on $\mc N_{\la}^{+}$ and $\mc N_{\la}^{-}$ respectively. Then
  $u \geq \phi$ and $v\ge \phi$ in $\Om$.
 \end{Lemma}
\begin{proof}
By Lemma \ref{lem2.1},  let $\{w_k\}$ be a sequence in $X_0$ such that supp$(w_k)$ is compact, $0 \leq w_k \leq (\phi - u)^+$ for each $k$ and $\{w_k\}$ strongly converges to $(\phi - u)^+$ in $X_0$. Then
\begin{equation}\label{eq3.5}
\frac{d}{dt}(\la t^{-q}+t^{\alpha})= -q \la t^{-q-1} +\alpha t^{\alpha-1} \leq 0
\;\text{if and only if}\; t^{\alpha+q}\leq \la  \left(\frac{q}{\alpha}\right).
\end{equation}
Using Lemma Lemma \ref{lem3.6} and \eqref{eq3.4}, we have
\begin{align*} &\int_Q \frac{\left(f(u)-f(\phi)\right)}{|x-y|^{n+sp}} (w_k(x) - w_k(y))~dxdy -  \int_\Om (\la u^{-q} + u^{\alpha})w_kdx + \int_{\Om} (\la \phi^{-q} +\phi^{\alpha})w_kdx \geq 0,\end{align*}
where $f(\xi)= |\xi(x)-\xi(y)|^{p-2}(\xi(x)-\xi(y))$.
 Since $\{w_k\}$ converges to $(\phi - u)^+$ strongly, we get a subsequence of $\{w_k\}$ such that $w_k(x)\rightarrow (\phi-u)^+(x)$ pointwise almost everywhere in $\Om$ and we write $w_k(x) = (\phi - u)^+(x)+ o(1)$ as $k \rightarrow \infty$. Then,
\begin{equation*}
\begin{split}
\int_Q \frac{\left(f(u)-f(\phi)\right)}{|x-y|^{n+sp}} (w_k(x) - w_k(y)) ~dxdy
&= \int_Q \frac{\left(f(u)-f(\phi)\right)}{|x-y|^{n+sp}} ((\phi - u)^+(x) - (\phi - u)^+(y))~dxdy \\
&\quad \quad + o(1)\int_Q \frac{\left(f(u)-f(\phi)\right)}{|x-y|^{n+sp}}~dxdy.
\end{split}
\end{equation*}
Further  we can see that
\begin{align}\label{eq3.6}
\int_Q &\frac{\left(f(u)-f(\phi)\right)}{|x-y|^{n+sp}} ((\phi - u)^+(x) - (\phi - u)^+(y))~dxdy\nonumber \\
& = \left(\int_{\Om_{1}\times \Om_1} + \int_{\Om_1 \times\Om_{2}} +\int_{\Om_{2}\times \Om_{1}} +\int_{\Om_{2}\times \Om_{2}}\right) \frac{\left(f(u)-f(\phi)\right)}{|x-y|^{n+sp}}((\phi - u)^+(x) - (\phi - u)^+(y)) ~dxdy
\end{align}
where
$\Om_{1}=\{x : \phi(x) \geq u(x)\}$ and $\Om_2= \{x :\phi(x) \leq u(x) \}$. Now, we separately estimate each integrals and to begin with, firstly we see that
\begin{equation}\label{eq3.7}
\int_{\Om_2 \times \Om_2} \frac{\left(f(u)-f(\phi)\right)}{|x-y|^{n+sp}}((\phi - u)^+(x) - (\phi - u)^+(y)) ~dxdy    =0.
\end{equation}
Next, we see that
\begin{align}\label{eq3.8}
&\int_{\Om_1\times \Om_1}  \frac{\left(f(u)-f(\phi)\right)}{|x-y|^{n+sp}}((\phi - u)^+(x) - (\phi - u)^+(y)) ~dxdy\nonumber \\
& = - \int_{\Om_1 \times \Om_1}\frac{\left(f(\phi)-f(u)\right)}{|x-y|^{n+sp}} ((\phi - u)(x) - (\phi - u)(y))~dxdy\nonumber\\
& \leq -\frac{1}{2^{p-2}} \int_{\Om_1\times \Om_1} \frac{|(\phi-u)(x)-(\phi-u)(y)|^{p}}{|x-y|^{n+sp}}~dxdy
\end{align}
using $|a-b|^p \leq 2^{p-2}(|a|^{p-2}a -|b|^{p-2}b)(a-b), ~ p\geq 2$ and $a,b \in \mb R$. Now, consider
\begin{equation}\label{eq3.9}
\begin{split}
& \int_{\Om_1\times \Om_2} \frac{\left(f(u)-f(\phi)\right)}{|x-y|^{n+sp}}((\phi - u)^+(x) - (\phi - u)^+(y)) ~dxdy\\
& =  \int_{\Om_1\times \Om_2} \frac{\left(f(u)-f(\phi)\right)}{|x-y|^{n+sp}}(\phi-u)(x) ~dxdy\\
& \leq -\frac{1}{2^{p-2}} \int_{\Om_1\times \Om_2} \frac{|(\phi-u)(x)-(\phi-u)(y)|^{p}}{|x-y|^{n+sp}}~dxdy +\\
& \quad + \int_{\Om_1\times \Om_2}\frac{\left(f(u)-f(\phi)\right)}{|x-y|^{n+sp}}(\phi-u)(y) ~dxdy
\end{split}
\end{equation}
and similarly, we will get
\begin{equation}\label{eq3.10}
\begin{split}
&\int_{\Om_2\times \Om_1} \frac{\left(f(u)-f(\phi)\right)}{|x-y|^{n+sp}}((\phi - u)^+(x) - (\phi - u)^+(y)) ~dxdy\\
 &\leq -\frac{1}{2^{p-2}} \int_{\Om_2\times \Om_1} \frac{|(\phi-u)(x)-(\phi-u)(y)|^{p}}{|x-y|^{n+sp}}~dxdy
 - \int_{\Om_2\times \Om_1}\frac{\left(f(u)-f(\phi)\right)}{|x-y|^{n+sp}}(\phi-u)(x) ~dxdy.
\end{split}
\end{equation}
Thus using \eqref{eq3.6}-\eqref{eq3.10}, we get
\begin{equation*}
\begin{split}
\int_{Q} \frac{\left(f(u)-f(\phi)\right)}{|x-y|^{n+sp}}&((\phi - u)^+(x) - (\phi - u)^+(y)) ~dxdy\\
& \leq -\frac{1}{2^{p-2}}\|(\phi-u)\|^p + \int_{\Om_1\times \Om_2}\frac{\left(f(u)-f(\phi)\right)}{|x-y|^{n+sp}}(\phi-u)(y) ~dxdy\\
& \quad \quad- \int_{\Om_2\times \Om_1}\frac{\left(f(u)-f(\phi)\right)}{|x-y|^{n+sp}}(\phi-u)(x) ~dxdy\\
& = -\frac{1}{2^{p-2}}\|(\phi-u)\|^p.
\end{split}
\end{equation*}
\noi Since $\phi^{\alpha +q} (x) \leq \la \left(\frac{q}{\alpha}\right)$, for each $x \in \Om$, using \eqref{eq3.5} we get
\begin{equation*}
\begin{split}
\int_\Om & ((\la u^{-q} + u^{\alpha})- (\la \phi^{-q} + \phi^{\alpha}))w_k dx \\
& = \int_{\Om \cap \{\phi \geq u\}} ((\la u^{-q} + u^{\alpha})- (\la \phi^{-q} + \phi^{\alpha}))(\phi-u)^{+}(x) dx + o(1) \geq 0\\
 \end{split}
\end{equation*}
which implies
 \begin{align*}
 0 &\leq  -\frac{1}{2^{p-2}}\|(\phi-u )^+\|^2 -\int_\Om (\la u^{-q} + u^{\alpha})w_kdx + \int_{\Om} (\la \phi^{-q} + \phi^{\alpha})w_kdx +o(1)\\
  &\leq  -\frac{1}{2^{p-2}}\|(\phi-u )^+\|^2 + o(1)
  \end{align*}
and letting $k \rightarrow \infty$, we get
$-\|(\phi-u )^+\|^2 \geq 0.$
Thus, we showed $u \geq \phi$. Similarly, we can show $v\ge \phi.$ \QED
\end{proof}

\section{Existence of minimizer on $\mc N_{\la}^{+}$ }
In this section, we will show that the minimum of $I$ on $\mc N_{\la}^{+}$ is achieved in $\mc N_{\la}^{+}.$ Also, we show that this minimizer is also solution of $(P_\la).$
\begin{Proposition}
For all $\la\in (0,\La)$, there exist $u_\la \in \mc N_{\la}^{+}$ satisfying $I(u_\la) = \inf\limits_{u\in \mc N_{\la}^{+}} I(u)$.
\end{Proposition}
\begin{proof}
Assume $0<q\leq1$ and $\la \in (0, \Lambda)$. We show that there exist $u_\la \in \mc N_{\la}^{+}$ such that $\ds I(u_\la) = \inf_{u\in \mc N_{\la}^{+}} I(u)$. Let $\{u_{k}\} \subset \mc N_{\la}^{+}$ be a sequence such that $I(u_{k}) \rightarrow \inf I(\mc N_{\la}^{+})$ as $k \rightarrow \infty$. Now by \eqref{eq3n1} we can assume that there exists $u_\la \in X_0$ such that $u_{k} \rightharpoonup u_\la$ weakly  in $X_0$ (up to subsequence). First we will show that $\inf I(\mc N_{\la}^{+}) < 0$. Let $u_0 \in \mc N_{\la}^{+}$, we have $\phi^{\prime \prime}_{u_0}(1) >0$ which gives
\[ \left( \frac{p-1+q}{\alpha+q}  \right)\|u_0\|^p > \int_{\Om}|u_0|^{\alpha+1} dx .\]
Therefore, using $\alpha>p-1$ we obtain
\begin{equation*}
\begin{split}
I(u_0) & = \left( \frac{1}{p} - \frac{1}{1-q} \right) \|u_0\|^p + \left( \frac{1}{1-q} - \frac{1}{\alpha+1} \right)\int_{\Om} |u_0|^{\alpha+1}dx\\
& \leq -  \frac{(p+q-1)}{p(1-q)} \|u_0\|^p + \frac{(p+q-1)}{(\alpha+1)(1-q)} \|u_0\|^p= \left( \frac{1}{\alpha+1}-\frac{1}{p}\right)\left(\frac{p+q-1}{1-q}\right) \|u_0\|^p<0\\
\end{split}
\end{equation*}
\noi Case(I) ($\alpha < p^*_s-1$) Firstly, we claim that $u_{\la} \in X_{+,q}$. When $0 < q <1$ , if $u_{\la}=0$ then $0 = I(u_{\la}) \leq \underline{\lim} \; I(u_{k}) < 0$, which is a contradiction. In the case $q = 1$, the sequence $\left \{ \int_{\Om} \ln(|u_{k}|) \right \}$ is bounded, since the sequence $\{I(u_{k})\}$ and $\{\|u_{k}\|\}$ is bounded and using Fatou's Lemma and $\ln(|u_{k}|) \leq u_{k},$ for each $k$, we get
\[ -\infty < \overline{\lim_{k \rightarrow \infty}} \int_{\Om} \ln(|u_{k}|) dx \leq  \int_{\Om} \overline{\lim_{k \rightarrow \infty}} \ln(|u_{k}|)dx =  \int_{\Om} \ln(|u_{\la}|)dx. \]
which implies $u_{\la} \not\equiv 0$ and thus, in both cases we have shown $u_{\la} \in X_{+,q}$. We claim that $u_{k} \rightarrow u_{\la}$ strongly in $X_0$. Suppose not. Then, we may assume $ \|u_{k} - u_{\la}\| \rightarrow c >0$. Using Brezis-Lieb lemma and embedding results for $X_0$ in the subcritical case, we have
\begin{equation}\label{eq6}
 \lim_{k \rightarrow \infty}\phi^{\prime}_{u_{k}}(1) = \phi^{\prime}_{u_{\la}}(1)+ c^p
\end{equation}
which implies $\phi^{\prime}_{u_{\la}}(1) + c^p =0$, using $\phi^{\prime}_{u_{k}}(1) = 0$ for each $k$. Since $\la \in (0, \Lambda)$, there exist $0 < t_1 < t_2$ (by fibering map analysis) such that $\phi^{\prime}_{u_{\la}}(t_1) = \phi^{\prime}_{u_{\la}}(t_2) = 0$ and $t_1u_{\la} \in \mc N_{\la}^{+}$. By \eqref{eq6}, we have $\phi^{\prime}_{u_{\la}}(1) < 0$ which gives two cases : $1< t_1$ or $t_2 < 1$. When $t_1>1$, we have
\[ \inf I(\mc N_{\la}^{+}) = \lim I(u_{k})= I(u_{\la})+\frac{c^p}{p} = \phi_{u_{\la}}(1) +\frac{c^p}{p} > \phi_{u_{\la}}(1) > \phi_{u_{\la}}(t_1) \geq \inf I(\mc N_{\la}^{+}),\]
which is  a contradiction. Thus we have $t_2 < 1$. We set, for $t>0$, $f(t) = \phi_{u_{\la}}(t) + \frac{c^pt^p}{2}, t>0$. From \eqref{eq6}, we get $f^{\prime}(1) = 0$ and since $0< t_2<1$, $f^{\prime}(t_2)= t^{p-1}_2c^p >0 $. So, $f$ is increasing on $[t_2, 1]$ and we obtain
\[ \inf I(\mc N_{\la}^{+}) = I(u_{\la})+\frac{c^p}{p} = \phi_{u_{\la}}(1) +\frac{c^p}{p} =f(1) > f(t_2) > \phi_{u_{\la}}({t_2}) > \phi_{u_{\la}}({t_1}) \geq \inf I(\mc N_{\la}^{+}) ,\]
which gives a contradiction. Hence, $c = 0$ and thus,  $u_{k} \rightarrow u_{\la}$ strongly in $X_0$. Since  $\la \in (0, \Lambda)$, we have  $\phi^{\prime \prime}_{u_{\la}}(1) > 0$, so we obtain $u_{\la} \in \mc N_{\la}^{+}$ and $I(u_{\la}) = \inf I(\mc N_{\la}^{+}).$

\noi Case(II) ($\alpha = p^*_s-1$ and $0<q<1$) We set $w_k := u_k - u_\la$ and claim that $u_k \rightarrow u_\la$ strongly in $X_0$. Suppose $\|w_k\|^p \rightarrow c^p \neq 0$ and $\int_{\Om} |w_k|^{p^*_s}dx \rightarrow d^{p^*_s}$ as $k \rightarrow \infty$. Since $u_k \in \mc N^+_{\la}$, using Brezis-Lieb Lemma, we get
\begin{equation}\label{eq8}
0  = \lim_{k \rightarrow \infty} \phi^{\prime}_{u_k}(1) = \phi^{\prime}_{u_\la}(1)+c^p -d^{p^*_s}
\end{equation}
which implies
$$ \|u_\la\|^p+c^p = \la \int_{\Om}|u_\la|^{1-q}dx + \int_{\Om}|u_k|^{p^*_s}dx +d^{p^*_s}.$$
We claim that $u_\la \in X_{+,q}$. Suppose $u_\la\equiv 0$. If $0<q<1$ and $c=0$ then $0 > \inf I(\mc N^+_{\la}) = I(0)=0$, which is a contradiction and if $c\neq 0$ then
\begin{equation}\label{eq4.2new}
\inf I(\mc N^+_{\la})= I(0)+\frac{c^p}{p} - \frac{d^{p^*_s}}{p^*_s} = \frac{c^p}{p} - \frac{d^{p^*_s}}{p^*_s} .\end{equation}
But we have $\|u_k\|^p_{p^*_s}S \leq \|u_k\|^p$ which gives $c^p \geq Sd^p $. Also from \eqref{eq8}, we have $c^p=d^{p^*_s}$. Then \eqref{eq4.2new} implies
\[ 0 > \inf I(\mc N^+_{\la}) = \left(\frac{1}{p}-\frac{1}{p^*_s}\right)c^p \geq \frac{s}{n}S^{\frac{n}{sp}},\]
which is again a contradiction. In the case $q = 1$, the sequence $\left \{ \int_{\Om} \ln(|u_{k}|) \right \}$ is bounded, since the sequence $\{I(u_{k})\}$ and $\{\|u_{k}\|\}$ is bounded, using Fatou's Lemma and $\ln(|u_{k}|) \leq u_{k},$ for each $k$, we get
\[ -\infty < \overline{\lim_{k \rightarrow \infty}} \int_{\Om} \ln(|u_{k}|) dx \leq  \int_{\Om} \overline{\lim_{k \rightarrow \infty}} \ln(|u_{k}|)dx =  \int_{\Om} \ln(|u_\la|)dx. \]
which implies $u_\la \not\equiv 0$ . Thus, in both cases we have shown that $u_\la \in X_{+,q}$. So, there exists $0 < t_1 < t_2$ such that $\phi^{\prime}_{u_{\la}}(t_1)= \phi^{\prime}_{u_{\la}}(t_2) = 0$ and $t_{1}u_{\la} \in \mc N^{+}_{\la}$. Then, three cases arise: \\
(i) $t_2 < 1$,\\
(ii) $t_2 \geq 1$ and $\frac{c^p}{p}- \frac{d^{p^*_s}}{p^*_s} < 0$, and \\
(iii) $t_2 \geq 1$ and $\frac{c^p}{p}- \frac{d^{p^*_s}}{p^*_s} \geq 0$.\\
\noi Case (i) Let $h(t) = \phi_{u_{\la}}(t)+ \frac{c^pt^p}{p} - \frac{d^{p^*_s}t^{p^*_s}}{p^*_s}$ for $t >0$. By \eqref{eq8} we get $ h^{\prime}(1) = \phi^{\prime}_{u_{\la}}(1)+c^p-d^{p^*_s} = 0$ and
\begin{equation*}
 h^{\prime}(t_2) = \phi^{\prime}_{u_{\la}}(t_2)+t^p_2c^p-t_2^{p^*_s}d^{p^*_s} = {t_2}^{p}(c^p - t_2^{p^*_s-p}d^{p^*_s}) > t_2^{p}(c^p - d^{p^*_s})
 > 0
\end{equation*}
which implies that $h$ increases on $[t_2,1]$. Then we get
\begin{equation*}
\begin{split}
\inf I(\mc N^+_{\la}) &= \lim I(u_k) \geq \phi_u(1) + \frac{c^p}{p}- \frac{d^{p^*_s}}{p^*_s}  = h(1) > h(t_2)\\
& =\phi_u(t_2) + \frac{c^pt_{2}^{p}}{p}- \frac{d^{p^*_s}t_{2}^{p^*_s}}{p^*_s} \geq \phi_u(t_2) + \frac{t_{2}^{p}}{p} (c^p - d^{p^*_s})\\
& > \phi_u(t_2) > \phi_u(t_1) \geq \inf I(\mc N^+_{\la}),
\end{split}
\end{equation*}
which is a contradiction.\\
\noi Case (ii) In this case, since $\la \in (0, \Lambda)$, we have $(c^p/p - d^{p^*_s}/{p^*_s}) < 0$ and $Sd^p \leq c^p$. Also we see that, for each $u_0 \in \mc N^{+}_{\la}$
\begin{equation*}
\begin{split}
0  < \phi^{\prime \prime}_{u_0}(1)& =(p-1) \|u_0\|^p + q \la \int_{\Om} |u_0|^{1-q} dx - (p^*_s-1) \int_{\Om}|u_0|^{p^*_s}dx \\
 & = (p-1+q) \|u_0\|^p +  (-q-p^*_s+1) \int_{\Om} |u_0|^{p^*_s} dx  \\
\end{split}
\end{equation*}
\noi  which implies $ (p-1+q)\|u_0\|^p  > (q+p^*_s-1)\int_{\Om}|u_0|^{p^*_s}dx =  (q+p^*_s-1)|u_0|^{p^*_s}_{p_s^*} $ \\
\noi or, $ C_{p^*_s}  \leq \left(\frac{p-1+q}{q+p^*_s-1}\right) \|u_0\|^{p-p^*_s}$
\noi or, $ \|u_0\|^{p}  \leq  \left(\frac{p-1+q}{q+p^*_s-1} \right )^{\frac{p}{p^*_s-p}} S^{\frac{p^*_s}{p^*_s-p}}$. Thus, we have
\[ \sup\{ \|u\|^p: u \in \mc N^+_{\la}\} \leq \left(\frac{p}{p^*_s}\right)^{\frac{p}{p^*_s-p}} S^{\frac{p^*_s}{p^*_s-p}} < c^p \leq \sup\{ \|u\|^p: u \in \mc N^+_{\la}\},\]
which gives a contradiction. Consequently, in case (iii)  we have
$$ \inf I(\mc N^+_{\la}) = I(u_\la)+\frac{c^p}{p}-\frac{d^{p^*_s}}{p^*_s} \geq I(u_\la) = \phi_{u_\la}(1) \geq \phi_{u_\la}(t_1) \geq \inf I(\mc N^+_{\la}) .$$
Clearly, this holds when $t_1 = 1$ and $(c^p/p - d^{p^*_s}/{p^*_s}) = 0$ which yields $c=0$ and $u_\la \in \mc N^+_{\la}$. Thus, $u_k \rightarrow u_\la$ strongly in $X_0$ as $k \rightarrow \infty$ and $I(u_\la) = \inf I(\mc N^+_{\la})$. \QED
\end{proof}

\begin{Proposition}\label{prp4.2}
$u_\la$  is a positive weak solution of ($P_\la$).
\end{Proposition}
\begin{proof}
Let $\psi \in C^{\infty}_c(\Om)$. By Lemma \ref{le05}, since $\phi > 0$, we can find $ \beta >0$ such that $u_\la \geq \beta$ on support of $\psi$. Then $u_\la+\e \psi\ge0$, for small $\e$.  With similar reasoning as  in the proof of Lemma \ref{le03}, $ I(u_\la+\epsilon \psi) \geq I(u_\la)$ for sufficiently small $\epsilon >0$. Then we have
\begin{equation*}
\begin{split}
0 & \leq \lim \limits_{\epsilon \rightarrow 0} \frac{I(u_\la+\epsilon\psi) - I(u_\la)}{\epsilon}\\
&= \int_Q \frac{|u_\la(x)-u_\la(y)|^{p-2}(u_\la(x)-u_\la(y))(\psi(x)-\psi(y))}{|x-y|^{n+ps}}~ dxdy - \la \int_\Om u_{\la}^{-q}\psi dx - \int_\Om u_{\la}^{\alpha}\psi ~dx.
\end{split}
\end{equation*}
Since $\psi \in C^{\infty}_c(\Om)$ is arbitrary, we conclude that $u_\la$ is a positive weak solution of $(P_\la)$.\QED
\end{proof}
We recall the following comparison principle from \cite{EPL}.
\begin{Lemma}\label{lem4.3}
Let $u,v \in X_0$ are such that $u\ge v$ in $\mathbb R^n\backslash\Om$ and
\[\int_{Q} \left(|u(x)-u(y)|^{p-2}(u(x)-u(y)) - |v(x)-v(y)|^{p-2}(v(x)-v(y))\right)\frac{(\psi(x)-\psi(y))}{|x-y|^{n+ps}}~dxdy\ge 0\]
for all non-negative $\psi \in X_0$.
Then $u\ge v$ in $\Om$.
\end{Lemma}
\begin{proof}
Proof follows by taking $\psi=(v-u)^+$  and using the equality
\[|b|^{p-2}b -|a|^{p-2} a = (p-1) (b-a) \int_0^1 |a+t(b-a)|^{p-2} dt.\] \QED
\end{proof}
\noi As a consequence, we have
\begin{Lemma}\label{lem4.4}
$\La_1 <\infty.$
\end{Lemma}

\begin{proof} Suppose $\La_1=\infty$. Then from Proposition \ref{prp4.2}, $(P_\la)$ has a solution for all $\la$. Now choose $\la$ large enough such that
\[\la t^{-q} + t^{p_{s}^{*}-1} > (\la_1 +\e) t^{p-1}, \; \text{for all}\; t\in (0,\infty).\]
Then $\oline{u}:=u_\la$ is a super solution of the eigenvalue problem
\[  (P_{\e})\quad u\in X_0; \; \text{and}\;
(-\De_p)^s u = (\la_1 +\e) |u|^{p-2} u \; \text{in}\; \Om.\]
Also we can choose $r$ small such that $\uline{u}:=r\phi_1$ is a subsolution of  $(P_{\e})$. Then by the boundedness of $u_\la$ (see Theorem \ref{reg1}) and $\phi_1$, we can choose $r$ small such that $\uline{u}\le \overline{u}.$

\noi Now, we consider the monotone iterations
\begin{align*}
u_0 = &r \phi_1\\
u_n\in  X_0; \; \text{and}\;
(-\De_p)^s u_n &= (\la_1 +\e) |u_{n-1}|^{p-2} u_{n-1} \; \text{in}\; \Om.
\end{align*}
Then by the weak comparison Lemma \ref{lem4.3}, we get
\[r\phi_1(x)\le u_1(x)\le u_2(x)\le...\le u_{n-1}(x)\le u_n(x)\le ....\le u_\la(x), \; \forall x\in \Om\]
Therefore, the sequence $\{u_n\}$ is bounded in $X_0$ and hence has a weakly convergent subsequence $\{u_n\}$ that converges to $u_0$. Thus, $u_0$ is a solution of $(P_{\e})$. Since $\e>0$ is arbitrary, we get a contradiction to the simplicity and isolatedness of $\la_1$.
\end{proof}

\section{Existence of minimizer on $\mc N^{-}_{\la}$}
In this section we show the existence of second solution for $(P_\la)$ in the subcritical case. We assume $\alpha<p_s^*-1$.
\begin{Proposition}
For all $\la\in (0,\La)$, there exist $v_\la \in \mc N_{\la}^{-}$ satisfying $I(v_\la) = \inf\limits_{v\in \mc N_{\la}^{-}} I(v)$.
\end{Proposition}
\begin{proof}
 Assume $0<q\leq1$ and $\la \in (0, \Lambda)$. We will show that there exists $v_{\la} \in \mc N_{\la}^{-}$ with $I(v_{\la}) = \inf I(\mc N_{\la}^{-})$.  Let $\{v_{k}\} \subset \mc N_{\la}^{-}$ be a sequence such that $\lim\limits_{k \rightarrow \infty}I(v_{k}) = \inf I(\mc N_{\la}^{-})$. Using Lemma \ref{le01}, we can assume that $v_{k} \rightharpoonup v_{\la}$ weakly as $k \rightarrow \infty$ in $X_0$. We claim that $v_{\la} \in X_{+,q}$. When $ 0<q<1 $, if $v_{\la}=0$ then $\{v_{k}\}$ converges strongly to $0$, which contradicts Lemma \ref{le01}. If $q=1$, we similarly have $- \infty < \int_{\Om} \ln(|v_{k}|) dx$ as above. So, by both the cases, we get $v_{\la} \in X_{+,q}$. Next, we claim that $\{ v_{k} \}$ converges strongly to $v_{\la}$ in $X_0$. Suppose not. Then we may assume $ \|v_{k} - v_{\la}\| \rightarrow d >0$, and we have
\begin{enumerate}
\item $\inf I(\mc N_{\la}^{-}) = \lim I(v_{k}) \geq  I(v_{\la}) + d^p/p$.
\item For each $k$,  $\phi^{\prime}_{v_{k}}(1) = 0$ and $\phi^{\prime \prime }_{v_{k}}(1) < 0\implies   \phi^{\prime}_{v_{\la}}(1) + d^p = 0 $ and $\phi^{\prime \prime}_{v_{\la}}(1) + d^p \leq 0 $.
\end{enumerate}
By (2), we have $\phi^{\prime}_{v_{\la}}(1) <0 $ and  $\phi^{\prime \prime}_{v_{\la}}(1) < 0$. So, there exists $ t_2 \in (0,1)$ such that $ \phi^{\prime}_{v_{\la}}(t_2) = 0$ and $\phi^{\prime \prime}_{v_{\la}}(t_2) < 0$. Thus, $t_2v_{\la} \in \mc N_{\la}^{-}$. Define $ g : {\mb R}^+ \rightarrow \mb R$ as $g(t) = \phi_{v_{\la}}(t) + \frac{d^pt^p}{2}$, for $ t>0$. From (2), we get $g^{\prime}(1) = 0$ and since $0 < t_2 < 1$, $g^{\prime}(t_2) = d^pt^{p-1}_2 > 0$. Then, $g$ is increasing on $[t_2, 1].$ Now we obtain
\[ \inf I(\mc N_{\la}^{-}) \geq  I(v_{\la})+\frac{d^p}{p} = \phi_{v_{\la}}(1) +\frac{d^p}{p} = g(1) \geq g(t_2) > \phi_{v_{\la}}(t_2) = I(t_2{v_{\la}}) \geq \inf I(\mc N_{\la}^{-}), \]
which gives a contradiction.  Hence, $d = 0$ and thus,  $\{ v_{k} \}$ converges strongly to $v_{\la}$ in $X_0$. Since  $\la \in (0, \Lambda)$, we have  $\phi^{\prime \prime}_{v_{\la}}(1) < 0$. Therefore, we obtain $v_{\la} \in \mc N_{\la}^{-}$ and $I(v_{\la}) = \inf I(\mc N_{\la}^{-}).$ This completes the proof of this proposition in subcritical case.\QED
\end{proof}
\begin{Proposition}\label{prp5.4}
 For $\la \in (0,\La),$ $v_\la$ is a positive weak solution of ($P_\la$).
\end{Proposition}
\begin{proof}
Let $\psi \in C^{\infty}_c(\Om)$. Using Lemma \ref{le05}, since $\phi > 0$ in $\Om$, we can find $ \beta >0$ such that $v_\la \geq \beta$ on $supp(\psi)$. Also, $t_{\epsilon} \rightarrow 1$ as $\epsilon \rightarrow 0+$, where $t_{\epsilon}$ is the unique positive real number corresponding to $(v_\la+\epsilon \psi)$ such that $t_\epsilon (v_\la+\epsilon \psi) \in \mc N^{-}_{\la}$. Then, by Lemma \ref{le03} we have
\begin{equation*}
\begin{split}
0 & \leq \lim\limits_{\epsilon \rightarrow 0}\frac{I(t_\e(v_\la+\epsilon\psi)) - I(v_\la)}{\epsilon} \leq \lim \limits_{\epsilon \rightarrow 0}\frac{I(t_{\epsilon}(v_\la+\epsilon\psi)) - I(t_{\epsilon} v_\la)}{\epsilon}\\
& = \int_Q \frac{|v_{\la}(x)-v_{\la}(y)|^{p-2}(v_\la(x)-v_\la(y))(\psi(x)-\psi(y))}{|x-y|^{n+sp}} ~dxdy -  \int_\Om \left(\la v_{\la}^{-q}+ v_{\la}^{\alpha}\right) \psi dx.
\end{split}
\end{equation*}
Since $\psi \in C^{\infty}_c(\Om)$ is arbitrary, we conclude that $v_{\la}$ is positive weak solution of $(P_\la)$.\QED
\end{proof}
\noi {\bf Proof of Theorem \ref{thm2.4}:} Proof follows from Proposition \ref{prp4.2} and Proposition \ref{prp5.4}. \QED

\begin{Remark}  To prove the existence of second positive solution in the critical case, one requires to know the classification of exact solutions of the problem
\[(-\De_p)^su = |u|^{p_s^*-2} u \text{ in } \mathbb R^n.\] These are the minimizers of $S$, the best constant of the embedding $X_0$ into $L^{p_s^*}$.  In \cite{s3, bms}, authors obtained several estimates on these minimizers and  conjectu that the solutions are dilations and translations of the radial  function
\[U(x)= \frac{1}{(1+|x|^{p'})^{(N-sp)/p}}, \; x\in \mathbb R^n\]
where $p'=\frac{p}{p-1}$.  In case of $p=2,$  these classifications are proved in \cite{rosen}, where author proved that all solutions are classified by dilations and translations of $U(x)$. Using these classifications, in \cite{TS} it is shown that
\[\sup\{I(u_\la+t U_\e)\; :\; t\ge 0\} < I(u_\la) + \frac{s}{n}S^{\frac{n}{2s}}.\]
where $U_\e= \e^{-(n-2s)/2} U(\frac{x}{\e}), \; x\in \mathbb R^n, \; \e>0$ and $u_\la$ is the minimizer on $\mc N_\la^+$.
Then by carefully analysing the related fiber maps it is shown that $u_\la + t U_\e \in \mc N_\la^-$, for large $t$. From this it follows
\[\inf I(\mc N_\la^-)< I(u_\la) + \frac{s}{n}S^{\frac{n}{2s}}\]
 Then the existence of minimizer   is shown using the analysis of fibering maps in Lemma \ref{lem3.2}.
\end{Remark}

\section{Regularity of  weak solutions}

In this section, we shall prove some regularity properties of positive weak solutions of $P_{\la}$. We begin with the following lemma.

\begin{Lemma}\label{weaksoldef}
Suppose $u$ is a  weak solution of $(P_{\la})$, then for each $w \in X_0$, it satisfies $u^{-q}w \in L^{1}(\Om) $ and
\begin{equation}\label{defweak}
\int_Q \frac{|u(x)-u(y)|^{p-2}(u(x)-u(y))(w(x)-w(y))}{|x-y|^{n+sp}} ~dxdy -  \int_\Om \left(\la  u^{-q} +  u^{\alpha}\right)w dx = 0. \end{equation}
\end{Lemma}
\begin{proof}
Let $u$ be a  weak solution of $(P_{\la})$ and $w \in X_{+}$. By Lemma \ref{lem2.1}, we get a sequence $\{w_k \} \in X_{0}$ such that $\{w_k\} \rightarrow w$ strongly in $X_0$, each $w_k$ has compact support in $
\Om$ and $0 \leq w_1 \leq w_2 \leq \ldots$. Since each $w_k$ has compact support in $\Om$ and $u$ is a positive weak solution of $(P_\la)$, for each $k$ we get
\[  \la \int_\Om u^{-q}w_k dx = \int_Q \frac{|u(x)-u(y)|^{p-2}(u(x)-u(y))(w_k(x)-w_k(y))}{|x-y|^{n+sp}} ~dxdy - \int_\Om u^{\alpha}w_k dx .\]
Using monotone convergence theorem, we get $u^{-q}w \in L^{1}(\Om)$  and
\[\la \int_\Om  u^{-q}w dx = \int_Q \frac{|u(x)-u(y)|^{p-2}(u(x)-u(y))(w(x)-w(y))}{|x-y|^{n+sp}} ~dxdy -\int_\Om u^{\alpha}w dx. \]
If $w \in X_0$ then $w = w^+ - w^-$ and $w^+, w^- \in X_{+}$. Since we proved the lemma for each $w \in X_+$, we obtain the conclusion.\QED
\end{proof}
Before proving our next result, let us recall some estimates or inequalities from \cite{bp}.
\begin{Lemma}\label{A.1}
Let $1<p<\infty$ and $f : \mathbb R \rightarrow \mathbb R$ be a $\text{C}^1$ convex function. If $\tau \geq 0$, $t, \; a, \; b \in \mathbb R$ and $A, \; B >0$ then
\[|f(a)-f(b)|^{p-2}(f(a)-f(b))(A-B) \leq |a-b|^{p-2}(a-b)(A|f^{\prime}(a)|^{p-2}f^{\prime}(a)- B|f^{\prime}(b)|^{p-2}f^{\prime}(b)).\]
\end{Lemma}

\begin{Lemma}\label{A.3}
Let $1<p<\infty$ and $g : \mathbb R \rightarrow \mathbb R$ be an increasing function, then we have
\[|G(a)-G(b)|^p \leq |a-b|^{p-2}(a-b)(g(a)-g(b))\]
where $G(t)= \int_0^t g^{\prime}(\tau)^{\frac{1}{p}} d\tau$, for $t \in \mathbb R$.
\end{Lemma}
\begin{Theorem}\label{reg1}
Let $u$ be a positive solution of $(P_{\la})$. Then $u \in L^{\infty}(\Om)$.
\end{Theorem}
\begin{proof}
Proof here is adopted from  Brasco and Parini \cite{bp}.
Let $\epsilon > 0$ be very small and define
\[f_{\epsilon}(t)= (\epsilon^2+t^2)^{\frac12}\]
which is smooth, convex and Lipschitz. Let $0< \psi \in C_c^{\infty}(\Om)$ and we take $\varphi = \psi|f_{\epsilon^{\prime}}(u)|^{p-2}f_{\epsilon}^{\prime}(u)$ as the test function in \eqref{defweak}. By taking the choices
\[ a=u(x),\;\; b=u(y),\;\; A=\psi(x),\;\; B=\psi(y) \]
in Lemma \ref{A.1}, we get
\begin{equation}
\int_Q \frac{|f_\epsilon(u(x))- f_{\epsilon}(u(y))|^{p-2}(f_\epsilon(u(x))-f_\epsilon(u(y)))(\psi(x)-\psi(y))}{|x-y|^{n+sp}} ~dxdy \leq  \int_\Om \left(|\la  u^{-q} +  u^{\alpha}|\right)|f_\epsilon^{\prime}(u)|^{p-1}\psi dx
\end{equation}
As $t \rightarrow 0$, $f_\epsilon(t) \rightarrow |t|$ and we have $|f_{\epsilon}^{\prime}(t)|\leq 1$. So using Fatou's Lemma, we let $\epsilon \rightarrow 0$ in above inequality and get
\begin{equation}\label{infty1}
\int_Q \frac{\mid|u(x)|- |u(y)|\mid^{p-2}(|u(x)|-|u(y)|)(\psi(x)-\psi(y))}{|x-y|^{n+sp}} ~dxdy \leq  \int_\Om \left(|\la  u^{-q} +  u^{\alpha}|\right)\psi~dx,
\end{equation}
for every $0<\psi \in C_c^{\infty}(\Om)$. The above inequality still holds for $0\leq \psi \in X_0$ ( similar proof as of Lemma \ref{weaksoldef}). Now, let us define $u_K= \min\{(u-1)^+, K\}\in X_0$, for $K>0$. For $\beta>0$ and $\rho>0$, we take $\psi = (u_K+ \rho)^{\beta}- \rho^{\beta}$ as test function in \eqref{infty1} and get
\begin{equation}\label{infty2}
\begin{split}
&\int_Q \frac{\mid|u(x)|- |u(y)|\mid^{p-2}(|u(x)|-|u(y)|)((u_K(x)+\rho)^{\beta}-(u_K(y)+\rho)^{\beta})}{|x-y|^{n+sp}} ~dxdy\\
 & \quad \quad\leq  \int_\Om \left(|\la  u^{-q} +  u^{\alpha}|\right)((u_K+\rho)^{\beta}- \rho^{\beta})~dx.
 \end{split}
\end{equation}
Then, by using Lemma \ref{A.3} with the function
\[g(u)= (u_K+\rho)^{\beta},\]
we get
\begin{equation}
\begin{split}
&\int_Q \frac{|(u_K(x)+\rho)^{\frac{\beta+p-1}{p}}- (u_K(y)+\rho)^{\frac{\beta+p-1}{p}}|}{x-y}^{n+sp}~dxdy\\
& \leq \quad \frac{(\beta+p-1)^p}{\beta p^p} \int_Q \frac{\mid|u(x)|- |u(y)|\mid^{p-2}(|u(x)|-|u(y)|)((u_K(x)+\rho)^{\beta}-(u_K(y)+\rho)^{\beta})}{|x-y|^{n+sp}} ~dxdy\\
&\leq \quad \quad  \frac{(\beta+p-1)^p}{\beta p^p} \int_\Om \la | u^{-q}|((u_K+\rho)^{\beta}- \rho^{\beta})~dx +  \int_\Om |u^{\alpha}|((u_K+\rho)^{\beta}- \rho^{\beta})~dx.
\end{split}
\end{equation}
Now, from the support of $u_K$ we have
\begin{equation}
\begin{split}
&\int_\Om \la | u^{-q}|((u_K+\rho)^{\beta}- \rho^{\beta})~dx +  \int_\Om |u^{\alpha}|((u_K+\rho)^{\beta}- \rho^{\beta})~dx\\
& = \int_{\{u\geq1\}} \la | u^{-q}|((u_K+\rho)^{\beta}- \rho^{\beta})~dx +  \int_{\{u \geq 1\} }|u^{\alpha}|((u_K+\rho)^{\beta}- \rho^{\beta})~dx\\
& \leq C_1 \int_{\{u\geq 1\}} (1+|u|^{\alpha}) ((u_K+\rho)^{\beta}- \rho^{\beta})~dx\\
& \leq 2C_1 \int_{\{u\geq 1\}}|u|^{\alpha} ((u_K+\rho)^{\beta}- \rho^{\beta})~dx\\
& \leq 2C_1 |u|^{\alpha}_{p_s^*}\; |(u_K+\rho)^{\beta}|_r
\end{split}
\end{equation}
where $C_1= \max \{\la,1\}$ and $r = \frac{p_s^*}{p_s^*-\alpha}$. By using Sobolev inequality given in Theorem 1 of \cite{maz}, we get
\begin{align*}
 \int_Q \frac{|(u_K(x)+\rho)^{\frac{\beta+p-1}{p}}- (u_K(y)+\rho)^{\frac{\beta+p-1}{p}}|}{{x-y}^{n+sp}}~dxdy
& \geq  \frac{1}{T_{p,s}} |(u_K+\rho)^{\frac{\beta+p-1}{p}}- \rho^{\frac{\beta+p-1}{p}}|_{p^*_s}^{p}\\
& \geq \frac{1}{T_{p,s}}\left( \left( \frac{\rho}{2}\right)^{p-1} |(u_K+\rho)^{\beta_p}|^{p}_{p_s^*} - \rho^{\beta+p-1}|\Om|^{\frac{p}{p_s^*}}\right),
\end{align*}
where $T_{p,s}$ is a nonnegative constant and the last inequality follows from triangle inequality and $(u_K+\rho)^{\beta+p-1}\geq \rho^{p-1}(u_K+\rho)^{\beta}$. Using all these estimates, we now have
\[|(u_K+\rho)^{\frac{\beta}{p}}|^p_{p_s^*} \leq C \left(T_{p,s} \left(\frac{2}{\rho}\right)^{p-1} \left( \frac{(\beta+p-1)^p}{\beta p^p}\right) |u|^{\alpha}_{p_s^*}\; |(u_K+\rho)^{\beta}|_r + \rho^{\beta}|\Om|^{\frac{p}{p_s^*}}\right),\]
where $C=C(p)>0$ is a constant. By convexity of the map $t \mapsto t^p$, we can show that
\[\frac{1}{\beta} \left(\frac{\beta+p-1}{p} \right)^p \geq 1.\]
Using this we can also check that
\[\rho^{\beta}|\Om|^{\frac{p}{p_s^*}} \leq \frac{1}{\beta} \left( \frac{\beta+p-1}{p}\right)^{p} |\Om|^{1-\frac{1}{r}- \frac{sp}{n}} |(u_K+\rho)^{\beta}|_r.\]
Hence we have
\begin{equation}
|(u_K+\rho)^{\frac{\beta}{p}}|^p_{p_s^*} \leq C \frac{1}{\beta} \left( \frac{\beta+p-1}{p}\right)^{p} |(u_K+\rho)^{\beta}|_r \left(
\frac{T_{p,s} |u|^{\alpha}_{p_s^*} }{\rho^{p-1}} + |\Om|^{1-\frac{1}{r}- \frac{sp}{n}} \right),
\end{equation}
for $C=C(p)>0$ is constant. We now suitably choose
\[ \rho = \left(T_{p,s} |u|^{\alpha}_{p_s^*} \right)^{\frac{1}{p-1}} |\Om|^{\frac{-1}{p-1}\left(1-\frac{1}{r}- \frac{sp}{n}\right)}\]
and let $\beta \geq 1$ be such that
\[ \frac{1}{\beta} \left(\frac{\beta+p-1}{p}\right)^{p} \leq \beta^{p-1}.\]
In addition, if we let $\tau = \beta r$ and $\nu= \frac{p_s^*}{pr} > 1$, then the above inequality uces to
\begin{equation}\label{iterate}
|(u_K+\rho)|_{\nu \tau} \leq \left( C |\Om|^{1-\frac{1}{r}- \frac{sp}{n}} \right)^{\frac{r}{\tau}} \left( \frac{\tau}{r}\right)^{\frac{(p-1)r}{\tau}} |(u_K+\rho)|_{\tau}
\end{equation}
At this stage itself, if we take $K \rightarrow \infty$, we can say that $(u-1)^+ \in L^{m}(\Om)$, for all $m$. This will imply that $u \in L^{m}(\Om)$, for all $m$. Now, we iterate \eqref{iterate} using $\tau_0 = r$ and
\[\tau_{m+1}= \nu \tau_m = \nu^{m+1}r\]
which gives
\begin{equation}\label{limit}
|(u_K+\rho)|_{\tau_{m+1}} \leq \left( C |\Om|^{1-\frac{1}{r}- \frac{sp}{n}} \right)^{\sum\limits_{i=0}^{m}\frac{r}{\tau_i}} \left( \prod_{i=0}^{m} \left(\frac{\tau_i}{r}\right)^{\frac{r}{\tau_i}}\right)^{p-1}|(u_K+\rho)|_r.
\end{equation}
Since $\nu >1 $,
\[ \sum\limits_{i=0}^{\infty}\frac{r}{\tau_i} = \sum\limits_{i=0}^{m} \frac{1}{\nu^i} = \frac{\nu}{\nu-1} \]
and
\[ \prod_{i=0}^{\infty}\left( \left(\frac{\tau_i}{r}\right)^{\frac{r}{\tau_i}}\right)^{p-1} = \nu^{\frac{\nu}{(\nu-1)^2}. }\]
Taking limit as $n \rightarrow 0$ in \eqref{limit}, we finally get
\[|u_K|_{\infty} \leq \left(C \nu^{\frac{\nu}{(\nu-1)^2}} \right)^{p-1} \left(|\Om|^{1-\frac{1}{r}- \frac{sp}{n}} \right)^{\frac{\nu}{\nu-1}} |(u_K+\rho)|_r.\]
Since $u_K \leq (u-1)^+$, using triangle inequality in above inequality we get,
\[|u_K|_{\infty} \leq C \left(\nu^{\frac{\nu}{(\nu-1)^2}} \right)^{p-1} \left(|\Om|^{1-\frac{1}{r}- \frac{sp}{n}} \right)^{\frac{\nu}{\nu-1}} \left( |(u-1)^+|_r + \rho |\Om|^{\frac{1}{r}} \right)\]
for some constant $C=C(p)>0$. If we now let $K \rightarrow \infty$, we get
\[ |(u-1)^+|_{\infty} \leq C \left(\nu^{\frac{\nu}{(\nu-1)^2}} \right)^{p-1} \left(|\Om|^{1-\frac{1}{r}- \frac{sp}{n}} \right)^{\frac{\nu}{\nu-1}} \left( |(u-1)^+|_r + \rho |\Om|^{\frac{1}{r}} \right).\]
Hence in particular, we say that $u \in L^{\infty}(\Om)$.\QED
\end{proof}

\begin{Theorem}
Let $u$ be a positive solution of $P_{\la}$. Then there exist $\gamma \in (0,s]$ such that $u \in C_{loc}^{\gamma}(\Om^{\prime})$, for all $\Om^{\prime} \subset \subset\Om$.
\end{Theorem}
\begin{proof}
Let $\Om^{\prime}\subset \subset \Om$. Then using lemma \ref{le05} and above regularity result, for any $\psi \in C_{c}^{\infty}(\Om)$ we get
\begin{equation*}
 \la \int_{\Om^{\prime}}u^{-q}\psi dx + \int_{\Om^{\prime}}u^{\alpha}\psi  dx\leq \la \int_{\Om^{\prime}}\phi_{1}^{-q}\psi  dx+ \|u\|_{\infty}^{\alpha} \int_{\Om^{\prime}}\psi dx \leq C \int_{\Om^{\prime}}\psi dx
\end{equation*}
for some constant $C>0$, since we can find $k>0$ such that $\phi_1>k$ on $\Om^{\prime}$. Thus we have $|(-\De_{p})^su|\leq C$ weakly on $\Om^{\prime}$. So, using theorem 4.4 of \cite{Asm} and applying a covering argument on inequality in corollary 5.5 of \cite{Asm}, we can prove that there exist $\gamma \in (0,s] $ such that $u \in C_{loc}^{\gamma}(\Om^{\prime})$, for all $\Om^{\prime} \Subset \Om$.\QED
\end{proof}


\section{Global existence of solution}
Let us define $\La = \sup \{\la > 0: (P_{\la}) \text{ has a solution}\} $.
\begin{Lemma}\label{lem7.2new}
$\La < +\infty.$
\end{Lemma}
\begin{proof}
The proof follows similarly as the proof of Lemma \ref{lem4.4}.\QED
\end{proof}

\noi In the following lemmas, we will show the existence of solution of $(P_{\la})$.

\begin{Lemma}\label{lem7.2}
If $\underline{u} \in X_0$ is a weak sub-solution and $\overline{u} \in X_0$ is a weak super-solution of ($P_{\la}$) such that $\underline{u} \leq \overline{u}$ a.e. in $\Om$, then there exists a weak solution $u \in X_{0}$ satisfying $\underline{u} \leq u \leq \overline{u}$.
\end{Lemma}
\begin{proof}
We follow \cite {yh}. We know that the functional $I$ is non- differentiable in $X_0$. Let $M:= \{u \in X_{0} : \underline{u} \leq u \leq \overline{u}\}$, then M is closed, convex and $I$ is weakly lower semicontinuous on $M$. We can see that if $\{u_k\} \subset M$ and $u_k \rightharpoonup u$ in $X_0$ as $k \rightarrow \infty$, we may assume $u_k \rightarrow u$ pointwise a.e. in $\Om$ (along a subsequence). Since $u \in M$, $\int_{\Om}|\overline{u}|^{\alpha+1} dx < + \infty$ and $\int_{\Om}|\overline{u}|^{1-q}  dx< + \infty$, then by Lebesgue Dominated Convergence theorem,
\[\int_{\Om}|u_k|^{\alpha+1}  dx \rightarrow \int_{\Om}|u|^{\alpha+1} dx\; \text{ and } \; \int_{\Om}|u_k|^{1-q} dx \rightarrow \int_{\Om}|u|^{1-q}dx.\]
So, ${\underline{\lim}}_{k\rightarrow \infty}I(u_k) \geq I(u)$. Thus, there exist $u \in M$ such that $I(u)= \inf_{u_0 \in M}I(u_0)$. We claim that $u$ is a weak solution of $(P_{\la})$. For $\e > 0$ and $\varphi \in X_0$, define $v_{\e} = u+\e\varphi-\varphi^{\e}+\varphi_{\e} \in M$ where $\varphi^{\e} =(u+\e\varphi-\overline{u})^+ \geq 0$ and $\varphi_{\e} =(u+\e\varphi-\underline{u})^- \geq 0$. For $t \in (0,1)$, $u+t(v_{\e}-u) \in M$ and we have
{\small \begin{equation*}
\begin{split}
0 & \leq \frac{I(u+t(v_{\e}-u))-I(u)}{t}\\
&= \lim_{t \rightarrow 0}\left( \frac{1}{pt}(\|u+t(v_{\e}-u)\|^p -\|u\|^p) +\la\int_{\Om}\frac{(G_q(u+t(v_{\e}-u))-G_q(u))}{t} dx\right.\\
& \left.\quad \quad \quad \quad - \frac{1}{\alpha+1}\int_{\Om} \frac{|u+t(v_{\e}-u)|^{\alpha+1}- |u|^{\alpha+1}}{t} dx \right)\\
& = \int_{Q}\frac{|u(x)-u(y)|^{p-2}(u(x)-u(y))((v_{\e}-u)(x)-(v_{\e}-u)(y))}{|x-y|^{n+sp}}~dxdy - \la \int_{\Om} u^{-q}(v_{\e}-u)dx\\
& \quad \quad \quad \quad - \int_{\Om}u^{\alpha}(v_{\e}-u) dx
\end{split}
\end{equation*}}
which gives
\begin{equation}\label{eq7.2}
\int_{Q}\frac{|u(x)-u(y)|^{p-2}(u(x)-u(y))(\varphi(x)-\varphi(y))}{|x-y|^{n+sp}}dxdy - \int_{\Om}(\la u^{-q}+u^{\alpha})\varphi dx \geq \frac{1}{\e}(H^{\e}-H_{\e})
\end{equation}
where
 \begin{align*}
 H^{\e} &= \int_{Q}\frac{|u(x)-u(y)|^{p-2}(u(x)-u(y))(\varphi^{\e}(x)-\varphi^{\e}(y))}{|x-y|^{n+sp}}~dxdy - \int_{\Om}(\la u^{-q}+u^{\alpha})\varphi^{\e} dx\\
 H_{\e} &= \int_{Q}\frac{|u(x)-u(y)|^{p-2}(u(x)-u(y))(\varphi_{\e}(x)-\varphi_{\e}(y))}{|x-y|^{n+sp}}~dxdy - \int_{\Om}(\la u^{-q}+u^{\alpha})\varphi_{\e} dx.
 \end{align*}
\noi Now we consider
 \begin{equation*}
\begin{split}
\frac{1}{\e}H^{\e}= \frac{1}{\e}\left(\int_{Q}\frac{|u(x)-u(y)|^{p-2}(u(x)-u(y))(\varphi^{\e}(x)-\varphi^{\e}(y))}{|x-y|^{n+sp}}~dxdy - \int_{\Om}(\la u^{-q}+u^{\alpha})\varphi^{\e}dx\right)
\end{split}
\end{equation*}
Let $\Om_1 = \{u+\e\varphi \geq \overline{u} >u\}$ and $\Om_2 = \{u+\e\varphi < \underline{u}\}$, then using the technique of Lemma \ref{le05}, we get
\begin{equation*}
\begin{split}
\frac{1}{\e} \int_{Q} &  \frac{|u(x)-u(y)|^{p-2}(u(x)-u(y))(\varphi^{\e}(x)-\varphi^{\e}(y))}{|x-y|^{n+sp}}~dxdy\\
&=\frac{1}{\e}\left(\int_{\Om_1\times\Om_1}+\int_{\Om_1\times\Om_2}+\int_{\Om_2\times\Om_1}\right)\frac{|u(x)-u(y)|^{p-2}(u(x)-u(y))
(\varphi^{\e}(x)-\varphi^{\e}(y))}{|x-y|^{n+sp}}dxdy\\
&=\frac{1}{\e}\int_{\Om_1\times\Om_1}\frac{|u(x)-u(y)|^{p-2}(u(x)-u(y))((u-\overline{u})(x)-(u-\overline{u})(y))}{|x-y|^{n+sp}}~dxdy\\
&\quad \quad +\int_{\Om_1\times\Om_1}\frac{|u(x)-u(y)|^{p-2}(u(x)-u(y))(\varphi(x)-\varphi(y))}{|x-y|^{n+sp}}~dxdy \\
&\quad \quad+\frac{1}{\e}\int_{\Om_1\times\Om_2}\frac{|u(x)-u(y)|^{p-2}(u(x)-u(y))}{|x-y|^{n+sp}}(u-\overline{u})(x)~dxdy\\
&\quad \quad+\int_{\Om_1\times\Om_2}\frac{|u(x)-u(y)|^{p-2}(u(x)-u(y))}{|x-y|^{n+sp}}\varphi(x)~dxdy\\
&\quad \quad-\frac{1}{\e}\int_{\Om_2\times\Om_1}\frac{|u(x)-u(y)|^{p-2}(u(x)-u(y))}{|x-y|^{n+sp}}(u-\overline{u})(y)~dxdy\\
&\quad \quad- \int_{\Om_2\times\Om_1}\frac{|u(x)-u(y)|^{p-2}(u(x)-u(y))}{|x-y|^{n+sp}}\varphi(y)~dxdy\\
&\geq\frac{3}{\e 2^{p-2}}\int_{\Om_1\times \Om_1} \frac{|(u-\overline{u})(x)-(u-\overline{u})(y)|^p}{|x-y|^{n+sp}}~dxdy \\
& \quad\quad+ \int_{\Om_1\times\Om_1}\frac{|u(x)-u(y)|^{p-2}(u(x)-u(y))(\varphi(x)-\varphi(y))}{|x-y|^{n+sp}}~dxdy\\
& \geq \int_{\Om_1\times\Om_1}\frac{|u(x)-u(y)|^{p-2}(u(x)-u(y))(\varphi(x)-\varphi(y))}{|x-y|^{n+sp}}~dxdy
\end{split}
\end{equation*}
where we used the inequality $|a-b|^p \leq 2^{p-2}(|a|^{p-2}a -|b|^{p-2}b)(a-b)$, for $ p\geq 2$ and $a,b \in \mb R$. Thus,
\begin{equation*}
\begin{split}
\frac{1}{\e}H^{\e}& \geq \int_{\Om_1\times \Om_1}\frac{|u(x)-u(y)|^{p-2}(u(x)-u(y))(\varphi(x)-\varphi(y))}{|x-y|^{n+sp}}~dxdy -\int_{\Om_1}(\la u^{-q}+u^{\alpha})\varphi^{\e}dx\\
& \geq \int_{\Om_1\times \Om_1}\frac{|u(x)-u(y)|^{p-2}(u(x)-u(y))(\varphi(x)-\varphi(y))}{|x-y|^{n+sp}}~dxdy - \int_{\Om_1}|\la \overline{u}^{-q}-{u}^{-q}||\varphi|dx\\
&=o(1)
\end{split}
\end{equation*}
as $\e \rightarrow 0$, since meas$(\Om_1) \rightarrow 0$ as $\e \rightarrow 0$. Similarly, as $\e \rightarrow 0$ we can show
that $\displaystyle \frac{1}{\e}H_{\e} \leq o(1).$\\
Therefore, from \eqref{eq7.2} taking $\e \rightarrow 0$, we get
\small{\[\int_{Q}\frac{|u(x)-u(y)|^{p-2}(u(x)-u(y))(\varphi(x)-\varphi(y))}{|x-y|^{n+sp}}~dxdy- \int_{\Om} \left(\la u^{-q}+u^{\alpha}\right)\varphi dx\geq o(1).\]}
Since $\varphi \in X_0$ is arbitrary, for all $\varphi\in X_{0}$ we get\\

$\displaystyle \int_{Q}\frac{|u(x)-u(y)|^{p-2}(u(x)-u(y))(\varphi(x)-\varphi(y))}{|x-y|^{n+sp}}~dxdy- \int_{\Om}(\la u^{-q}+u^{\alpha})\varphi dx=0. $\QED
\end{proof}
\begin{Proposition}\label{lem7.3}
For $\la \in (0,\La)$, ($P_{\la}$) has a weak solution $u_{\la} \in X_0$.
\end{Proposition}
\begin{proof}
We fix $\la \in (0, \La)$. By definition of $\La$, there exists $\la_{0} \in (\la , \La)$ such that $(P_{\la_0})$ has a solution $u_{\la_0}$ (say). Then $\overline{u} = u_{\la_0}$ becomes a super-solution of $(P_{\la})$.  Now consider the function $\phi_1$  as the eigenfunction of $(-\De_{p})^s$ corresponding to the smallest eigenvalue $\la_1$. Then $\phi_1 \in L^{\infty}(\Om)$ and
 \begin{equation*}
 \quad (-\De_{p})^s \phi_1 = \la_1 |\phi_1|^{p-2}\phi_1, \; \phi_1>0\; \text{in}\;
\Om,\quad   \phi_1 = 0 \; \mbox{on}\; \mb R^n \setminus\Om.
\end{equation*}
Let us choose $t >0$  such that $t\phi_1 \leq \overline{u}$ and $t^{p+q-1}\phi_1^{p+q-1} \leq \la/\la_1$. If we define $\underline{u}=t\phi_1$, then
\begin{equation*}
\begin{split}
(-\De_{p})^s{\underline{u}} &=  \la_1 t^{p-1}\phi_1^{p-1} \leq \la t^{-q}\phi_1^{-q}\\
& \leq \la t^{-q}\phi_1^{-q} + t^{\alpha}\phi_1^{\alpha} = \la \underline{u}^{-q}+\underline{u}^{\alpha}.
\end{split}
\end{equation*}
that is, $\underline{u}$ is a sub-solution of $(P_{\la_0})$ and $\underline{u} \leq \overline{u}$. Applying Lemma \ref{lem7.2} shows that $(P_{\la})$ has a solution for all $\la \in (0, \La)$. This completes the proof. \QED
\end{proof}
{\bf Proof of Theorem \ref{thm2.5}}: Proof follows from Proposition \ref{lem7.3} and Lemma \ref{lem7.2new}. \QED
\begin{Remark} We remark that the method in Lemma \ref{lem7.2} we can show the existence of solution for pure singular problem:
\begin{equation}\label{obs1}
\begin{split}
(-\De_p)^su = \la u^{-q} \text{ in } \Om, \quad
u = 0 \text{ in } \mb R^n \backslash \Om.
\end{split}
\end{equation}
where $0<q<1$.
We define $u$ to be a positive weak solution of (\ref{obs1}) if $u >0$ in $\Om$, $u \in X_0$ and
$$ \int_Q \frac{|u(x)-u(y)|^{p-2}(u(x)-u(y))(\psi(x)-\psi(y))}{|x-y|^{n+sp}} ~dxdy - \la \int_\Om  u^{-q} dx = 0 \;\text{for all} \; \psi \in X_0.$$
Also, we say $u \in X_0$ to be a positive weak sub-solution of $(\ref{obs1})$ if $u >0$ and
$$ \int_Q \frac{|u(x)-u(y)|^{p-2}(u(x)-u(y))(\psi(x)-\psi(y))}{|x-y|^{n+sp}} ~dxdy \leq \la \int_\Om  u^{-q}  dx\;\text{for all} \; \psi \in X_0.$$
We define the functional $J_{\la} : X_{0} \rightarrow (-\infty, \infty]$ by
\[ J_{\la}(u) = \frac{1}{p} \int_Q \frac{|u(x) - u(y)|^p}{|x-y|^{n+sp}}~dxdy - \la \int_\Om  G_q(u) dx \]
where $G_q$ is as defined in section 2. One can easily see that $J_{\la}$ is coercive, bounded below and weakly lower semicontinuous in $X_0$. Thus there exist a $u_0 \in X_0$ such that $\inf_{u \in X_0} I(u)= I(u_0)$. We claim that $u_0$ is a positive weak solution of (\ref{obs1}). We choose $t >0$ such that $t\phi_1 \leq u_0$ in $\Om$ and $t\phi_1$ is a sub-solution of  (\ref{obs1}) ($\phi_1$ is defined in proposition \ref{lem7.3}). Let us define $M := \{u \in X_{0} : \underline{u} \leq u \}$, where $\underline{u}$ is a weak sub-solution of (\ref{obs1}). Then $u_0 \in M$ and following the proof of lemma \ref{lem7.2} with $v_\e = u_0+\e \varphi + \varphi_{\e}$ where $\e>0, \varphi_{\e}= (u_0+\e \varphi-\underline{u})^-$ and $ \varphi \in X_0$, we can show that $u_0$ is a positive weak solution of ($\ref{obs1}$).

\end{Remark}


\end{document}